\documentclass[review,hidelinks,onefignum,onetabnum]{siamart250211}

\usepackage{lipsum}
\usepackage{amsfonts}
\usepackage{graphicx}
\usepackage{epstopdf}
\usepackage{tikz}
\usetikzlibrary{arrows, shapes, decorations.pathreplacing, positioning}
\usepackage{amsmath}
\usepackage{algorithmic}
\ifpdf
  \DeclareGraphicsExtensions{.eps,.pdf,.png,.jpg}
\else
  \DeclareGraphicsExtensions{.eps}
\fi

\newsiamremark{remark}{Remark}
\newsiamremark{hypothesis}{Hypothesis}
\crefname{hypothesis}{Hypothesis}{Hypotheses}
\newsiamthm{claim}{Claim}
\newsiamremark{fact}{Fact}
\crefname{fact}{Fact}{Facts}

\headers{Parareal for Coupled Elliptic-Parabolic Problems}{Iñigo Jimenez-Ciga, Francisco Gaspar, Kundan Kumar, Florin A. Radu}

\title{Parareal Algorithm for Coupled Elliptic-Parabolic Problems\thanks{Submitted to the editors DATE.}}

\author{Iñigo Jimenez-Ciga\thanks{Department of Statistics, Computer Science and Mathematics, Public University of Navarre, Spain 
  (\email{inigo.jimenez@unavarra.es}.}
\and Francisco Gaspar\thanks{Department of Applied Mathematics, University of Zaragoza, Spain
  (\email{fjgaspar@unizar.es}.}
\and Kundan Kumar\thanks{Department of Mathematics, University of Bergen, Norway
  (\email{kundan.kumar@uib.no, florin.radu@uib.no}.} \and  Florin A. Radu\footnotemark[4]}

\usepackage{amsopn}

\ifpdf
\hypersetup{
  pdftitle={Parareal Algorithm for Coupled Elliptic-Parabolic Problems},
  pdfauthor={Iñigo Jimenez-Ciga, Francisco Gaspar, Kundan Kumar, Florin A. Radu}
}
\fi

\usepackage{tikz}
\usepackage[caption=false]{subfig}
\usepackage{algorithmic}
\Crefname{ALC@unique}{Line}{Lines}
\usetikzlibrary{shapes.geometric, arrows}

\tikzstyle{startstop} = [rectangle, rounded corners, minimum width=1cm, minimum height=1cm,text centered, draw=black, fill=red!30]
\tikzstyle{io} = [trapezium, trapezium left angle=70, trapezium right angle=110, minimum width=3cm, minimum height=1cm, text centered, draw=black, fill=blue!30]
\tikzstyle{process} = [rectangle, minimum width=1cm, minimum height=1cm, text centered, draw=black, fill=orange!30]
\tikzstyle{decision} = [diamond, minimum width=3cm, minimum height=1cm, text centered, draw=black, fill=green!30]
\tikzstyle{arrow} = [thick,->,>=stealth]

\begin{document}

\maketitle

\begin{abstract}
We present a convergence analysis of the parallel-in-time integration method known as the Parareal algorithm for degenerate differential-algebraic systems arising from quasi-static Biot models, which govern coupled flow and deformation in porous media. The underlying system exhibits a saddle-point structure and degeneracy due to the quasi-static assumption. We extend the Parareal algorithm to this setting and propose three coarse propagators: monolithic, fixed-stress, and multirate fixed-stress schemes. For each, we derive sufficient conditions for convergence and establish explicit time step restrictions that guarantee contractivity of the iteration matrix.   Numerical experiments show computational savings accrued by using a parareal solver in multiphysics simulations involving poroelasticity and other coupled systems.
\end{abstract}



\section{Introduction}
In the context of increasingly complex computer architectures, parallelization has become a fruitful strategy in order to reduce runtimes of simulations. Among parallel-in-time integrators, the Parareal algorithm is well-known for its simplicity and efficiency. This algorithm was first introduced by Lions, Maday and Turinici in \cite{foundational2001lions}, and subsequent equivalent formulations were proposed by \cite{foundational2002baffico} and \cite{BalMaday2002}. It consists of a parallel-in-time integration method that combines two numerical propagators with different computational characteristics. A coarse propagator provides inexpensive but less accurate approximations, while a fine propagator delivers accurate solutions albeit at greater computational cost. This decomposition enables parallel computation across temporal subintervals. A coarser time mesh is usually considered for the coarse propagator, whereas the fine propagator is implemented on a refined grid. The Parareal algorithm can be interpreted as a predictor-corrector method, as well as a multiple shooting method or a multigrid method (see \cite{scalaranalysis2007gander}).

Rigorous convergence analysis of the Parareal algorithm has been studied for the initial value problem with a symmetric positive definite matrix $A \in \mathbb{R}^{n \times n}$ given by
\begin{equation}
\frac{du}{dt} + Au = f(t), \quad u(0) = u_0, \quad t \in [0,T],
\label{eq:ivp}
\end{equation}
where $u(t) \in \mathbb{R}^n$ represents the solution vector and $f(t) \in \mathbb{R}^n$ is a known source term \cite{scalaranalysis2007gander}. The symmetric positive definiteness of $A$ (i.e., $A = A^\top$ and $\langle x, Ax \rangle > 0$ for all nonzero $x$) ensures the convergence of the Parareal algorithm.

In this work, we present a novel theoretical framework to study the convergence of the Parareal method for  degenerate differential-algebraic equations of the form
\begin{equation}
C \frac{d}{dt}
\begin{pmatrix}
u \\ p
\end{pmatrix}
+ A
\begin{pmatrix}
u \\ p
\end{pmatrix}
= 
\begin{pmatrix}
f \\ g
\end{pmatrix},
\label{eq:biot_system}
\end{equation}
where $C$ has a degenerate structure.  This type of differential-algebraic equations typically arises from the spatial discretization of coupled elliptic–parabolic problems. In \eqref{eq:biot_system}, $u$ denotes the unknown of the elliptic problem, while $p$ denotes the unknown of the parabolic equation. Representative examples of such coupled problems include poroelasticity (see, e.g., \cite{cheng2016poroelasticity} for geosciences applications, \cite{cowin1999bone} for biological applications), thermoelasticity \cite{nowacki2013thermoelasticity} (e.g., chapter 1) and electrochemical systems such as fuel cells and lithium-ion batteries \cite{saralee2016thesis}. Typically,  the capacity matrix $C$ contains zeros in its first block row and is given by
\begin{equation*}
C = \begin{pmatrix}
0 & 0 \\
C_{pu} & C_{pp} 
\end{pmatrix}.
\end{equation*}

The properties of the matrix $A$, written in block form by
\begin{equation*}
A = \begin{pmatrix}
A_{uu} & A_{up} \\
A_{pu} & A_{pp} 
\end{pmatrix},
\end{equation*}
depend on the specific application under consideration. The degeneracy of $C$ alters the system's dynamics compared to standard parabolic problems. The first block row enforces an instantaneous elliptic constraint on the unknown $u$, while the second row governs the evolution of $p$ through a parabolic equation modified by coupling terms. The coupled nature of the model problem and the degenerate structure of $C$  lead to an interplay between iterative schemes for the problem, time stepping schemes, and different choices for the propagators within the Parareal algorithm. The convergence of the resulting numerical schemes is by no means guaranteed. The theory developed in this work is not limited to this class of problems and applies to \cref{eq:biot_system} with arbitrary matrix  $C$.

Regarding the sequential solution of \cref{eq:biot_system}, three different approaches are commonly used to solve elliptic–parabolic problems. The first class of methods includes the so-called fully implicit, or monolithic, methods, in which all unknowns are solved simultaneously. The main drawback of these schemes is their high computational cost. The second class comprises iterative coupling schemes, in which either the parabolic or the elliptic problem is solved first, followed by the solution of the other subproblem. This procedure is repeated until a converged solution is obtained within a prescribed tolerance. These methods are attractive in practice, since they involve the solution of smaller linear systems for which well-known and efficient linear solvers can be employed.  Finally, the third class includes explicit, or non-iterative sequential, schemes, in which the system is decoupled and no iterations between the subproblems are required. The main advantages of these methods are their low computational cost and ease of implementation compared with the previous approaches, possibly at the expense of reduced numerical accuracy.  Abundant literature exists on these three types of methods for the poroelasticity problem (see \cite{settari1998coupled}). Examples of well-known iterative coupling methods include the undrained  and fixed-stress splitting methods \cite{reveron2021iterative, CastellettoWhiteTchelepi2015, KimTchelepiJuanes2011, mikelic2013}.

Usually, the elliptic equation exhibits a slower time scale than the parabolic equation. For example, in the context of poroelasticity, the mechanical component evolves more slowly than the flow component, whereas in electrochemical processes at the cathode, the electric potential equation is associated with a  faster time scale than the ionic diffusion equation. This separation of time scales can be exploited by multirate time-integration schemes, in which multiple fine time steps are taken for the parabolic equation within a single coarse time step for the elliptic equation. The multirate approach is motivated from the field of ordinary differential equations \cite{savcenco2007multirate}. In other multiphysics applications, for example, involving Stokes-Darcy type or other hyperbolic models, multirate approaches have been used in \cite{delpopolo2019conservative, rybak2014multiple, rybak2015multirate, ShanZhengLayton2013}. The splitting schemes applied to \eqref{eq:biot_system} provide a natural framework for multirate methods, as the decoupling of the two equations enables the use of different time steps for each equation. Similarly, independent spatial discretizations can be employed for the elliptic and parabolic equations, leading to multiscale approaches. In the context of the poroelasticity problem, multirate extensions of the undrained split and the fixed stress split schemes were proposed in the work of \cite{almani2016convergence, kumar2016multirate}, and nonlinear and multiscale extensions of the fixed stress split scheme were proposed in the work of \cite{borregales2018robust, kraus2024fixed} and \cite{dana2018convergence}. Moreover, the convergence of the undrained split iterative scheme in heterogenous poroelastic media was established in the work of \cite{almani2020convergence} for both the single rate and multirate schemes, and the convergence of the single rate fixed stress split scheme in heterogeneous poroelastic media was addressed in the work of \cite{almani2023convergence}.   In this work, we follow this approach for the more general degenerate elliptic-parabolic problem \cref{eq:biot_system}, that is, employing an extended formulation of both well-known fixed-stress and multirate fixed-stress iterative schemes. We then apply the theory developed here to study the convergence of the Parareal method combined with the aforementioned integration schemes, both analytically and numerically.

The remainder of the paper is structured as follows. \Cref{sec:parareal} contains a brief description of the parallel-in-time Parareal algorithm, and new convergence conditions are stated in a more general setting. \Cref{sec:biot} is devoted to the formulation of the quasi-static Biot's consolidation model, which motivates the extension of the well-known fixed-stress scheme and the multirate fixed-stress iterative method for the coupled elliptic-parabolic problem \cref{eq:biot_system}. The section also provides a thorough description of both methods. \Cref{sec:contraction} includes the main results of the present work, providing theoretical constrains for the time step size in order to ensure convergence of the Parareal algorithm for the monolithic approach, the fixed-stress and multirate fixed-stress methods. Finally, \Cref{sec:experiments} illustrates the convergence behaviuour of the proposed solvers in a more realistic framework, and, in \Cref{sec:conclusions}, some conclusions are drawn and further extensions of this work are sketched.

\section{The Parareal algorithm} \label{sec:parareal}

This section is devoted to the formulation of the Parareal algorithm and the general conditions that the propagators are required to satisfy in order to guarantee the convergence of the method. First introduced in \cite{foundational2001lions}, the parallel-in-time integrator stands out for being extremely efficient dealing with parabolic problems. 

\subsection{Formulation}

In this work, we consider the formulation given by \cite{foundational2002baffico}, which is equivalent to the one proposed in the foundational paper \cite{foundational2001lions}. Let us consider two different partitions of the time interval $[0,T]$. First, let us define a uniform division into $N_c$ subintervals, denoted by $[T_n,T_{n+1}]$, for $n\in\{0,1,\dots,N_c-1\}$, and coarse time step $\Delta T = T/N_c$. Each coarse subinterval is subsequently divided uniformly into $N_f$ smaller intervals, denoted by $[t_{nN_f + j}, t_{nN_f+j+1}]$, for $j\in\{0,1,\dots,N_f-1\}$. Observe that $T_n = t_{nN_f}$ for all $n\in\{0,1,\dots,N_c-1\}$. The fine time step is then denoted by $\delta t = \Delta T / N_f = T/(N_cN_f)$. Figure \ref{fig:time_mesh} illustrates the two time meshes under consideration for the formulation of the Parareal algorithm. We remark that time meshes do not necessarily have to be uniform, in fact, there already exists an adaptive version of the Parareal algorithm (cf. \cite{adaptive2020maday}). However, for the sake of simplicity, this work will stick to this simplified version of the time discretization.

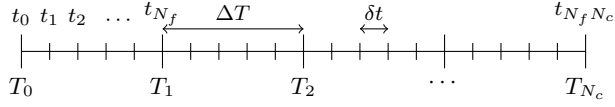
\begin{figure}[t]
\centering

        \begin{tikzpicture}[scale=0.75]
        \draw (0,0) -- (10,0);
        \draw (0,-0.3) -- (0,0.3);
        \draw (2.5,-0.3) -- (2.5,0.3);
        \draw (5,-0.3) -- (5,0.3);
        \draw (7.5,-0.3) -- (7.5,0.3);
        \draw (10,-0.3) -- (10,0.3);
        \draw (0.5,-0.15) -- (0.5,0.15);
        \draw (1,-0.15) -- (1,0.15);
        \draw (1.5,-0.15) -- (1.5,0.15);
        \draw (2,-0.15) -- (2,0.15);
        \draw (3,-0.15) -- (3,0.15);
        \draw (3.5,-0.15) -- (3.5,0.15);
        \draw (4,-0.15) -- (4,0.15);
        \draw (4.5,-0.15) -- (4.5,0.15);
        \draw (5.5,-0.15) -- (5.5,0.15);
        \draw (6,-0.15) -- (6,0.15);
        \draw (6.5,-0.15) -- (6.5,0.15);
        \draw (7,-0.15) -- (7,0.15);
        \draw (8,-0.15) -- (8,0.15);
        \draw (8.5,-0.15) -- (8.5,0.15);
        \draw (9,-0.15) -- (9,0.15);
        \draw (9.5,-0.15) -- (9.5,0.15);

        \draw [<->] (2.5,0.4) -- (5,0.4);
        \draw [<->] (6,0.4) -- (6.5,0.4);

        \path [anchor=south] (0,0.3) node {\footnotesize $t_0$};
        \path [anchor=south] (0.5,0.3) node {\footnotesize $t_1$};
        \path [anchor=south] (1,0.3) node {\footnotesize $t_2$};
        \path [anchor=south] (2.5,0.3) node {\footnotesize $t_{N_f}$};
        \path [anchor=south] (10,0.3) node {\footnotesize $t_{N_fN_c}$};
        \path [anchor=south] (3.75,0.4) node {\footnotesize $\Delta T$};
        \path [anchor=south] (6.25,0.4) node {\footnotesize $\delta t$};
        \path [anchor=south] (1.75,0.3) node {\footnotesize $\cdots$};
        \path [anchor=north] (0,-0.3) node {\small $T_0$};
        \path [anchor=north] (2.5,-0.3) node {\small $T_1$};
        \path [anchor=north] (5,-0.3) node {\small $T_2$};
        \path [anchor=north] (7.5,-0.3) node {\small $\cdots$};
        \path [anchor=north] (10,-0.3) node {\small $T_{N_c}$};
	\end{tikzpicture}
	\caption{Coarse and fine discretizations of the time interval $[0,T]$.}
	\label{fig:time_mesh}
\end{figure}

It is well known that the Parareal algorithm is based on two time integrators, the so-called coarse and fine propagators. The former propagator is cheap and inaccurate, and it runs sequentially over the coarse time mesh, whereas the latter performs in the fine mesh and is expensive but accurate. The algorithm permits the parallelization of the latter for every coarse time interval $[T_n,T_{n+1}]$. Note that both propagators are not inherently different. In fact, more often than not, a single time integrator is used together with different time meshes which provide more or less accuracy to the approximations.

Since we are considering the numerical integration of problem \cref{eq:biot_system}, the formulation of the Parareal algorithm is only described for systems of linear differential equations. For the sake of clarity, let us denote the vector containing both unknowns from the elliptic and parabolic equations by $\boldsymbol{v} = \begin{pmatrix}
    \boldsymbol{u}^\top & p^\top
\end{pmatrix}^\top$.

Then, the Parareal algorithm attempts to obtain a sufficiently accurate approximation of the solution of the system
\begin{equation} \label{eq:fine_problem}
    M_F \boldsymbol{V} = \boldsymbol{b}_{F},
\end{equation}
where
\begin{equation*}
    M_F = \begin{pmatrix}
        I & & & \\ -\Phi_F & I & & \\ & \ddots & \ddots & \\ & & -\Phi_F & I
    \end{pmatrix}, \quad \boldsymbol{V} = \begin{pmatrix}
        \boldsymbol{v}_0 \\ \boldsymbol{v}_1 \\ \vdots \\ \boldsymbol{v}_{N_c}
    \end{pmatrix}, \quad \text{and} \quad \boldsymbol{b}_F = \begin{pmatrix}
    \boldsymbol{v}_0 \\ \boldsymbol{b}_{F,1} \\ \vdots \\ \boldsymbol{b}_{F,N_c}
    \end{pmatrix},
\end{equation*}
taking advantage of the available parallel capabilities. Observe that the fully-discrete problem \cref{eq:fine_problem} denotes in fact the compact formulation of the fine propagator for all the interval $[0,T]$, such that the approximation $\boldsymbol{v}_{n+1}$ at time $t = T_{n+1}$ given by the fine propagator is computed as follows
\begin{equation} \label{eq:fine_step}
    \boldsymbol{v}_{n+1} = \Phi_F\boldsymbol{v}_n + \boldsymbol{b}_{F,n+1}.
\end{equation}
In the following, the fine approximation (also known as target solution), that is, the solution of \cref{eq:fine_problem}, is denoted by $\boldsymbol{V}^F$.

Analogously to \cref{eq:fine_step}, let us denote by $\Phi_G$ the matrix of the system resulting from the discretization in time by the coarse propagator for the integration over one coarse time interval. Then, the resulting system can be written as
\begin{equation*}
    \boldsymbol{v}_{n+1} = \Phi_G\boldsymbol{v}_n + \boldsymbol{b}_{G,n+1},
\end{equation*}
where $\boldsymbol{b}_{G,n+1}$ and $\boldsymbol{b}_{F,n+1}$ are possibly different. In addition, the whole system of equations for the integration by means of the coarse propagator can be written as $M_G\boldsymbol{V} = \boldsymbol{b}_{G}$, where
\begin{equation} \label{eq:notation_matrix2}
    M_G = \begin{pmatrix}
        I & & & \\ -\Phi_G & I & & \\ & \ddots & \ddots & \\ & & -\Phi_G & I
    \end{pmatrix} \quad \text{and} \quad \boldsymbol{b}_G = \begin{pmatrix}
    \boldsymbol{v}_0 \\ \boldsymbol{b}_{G,1} \\ \vdots \\ \boldsymbol{b}_{G,N_c}
    \end{pmatrix}.
\end{equation}

After introducing the notation employed in this section for the coarse and fine propagators, let us describe the formulation of the Parareal algorithm. Such an integrator is an iterative method that requires an initial guess. Among the possible options, we consider a coarse approximation computed sequentially over the whole time interval, that is, $\boldsymbol{v}_0^0 = \boldsymbol{v}_0$ and
\begin{equation*}
    \boldsymbol{v}_{n+1}^0 = \Phi_G\boldsymbol{v}_n^0 + \boldsymbol{b}_{G,n+1},
\end{equation*}
for all $n \in \{0,1,\dots,N_c-1\}$. The rest of the iterations $k\in\{0,1,\dots\}$ are given by the following expression: $\boldsymbol{v}_0^{k+1} = \boldsymbol{v}_0$ and
\begin{equation*}
    \boldsymbol{v}_{n+1}^{k+1} = \Phi_G \boldsymbol{v}_n^{k+1} + (\Phi_F - \Phi_G)\boldsymbol{v}_n^k + \boldsymbol{b}_{F,n+1},
\end{equation*}
for $n\in\{0,1,\dots,N_c-1\}$. Observe that the term $(\Phi_F - \Phi_G)\boldsymbol{v}_n^k$ can be computed independently for all $n\in\{0,1,\dots,N_c-1\}$, since the vector $\boldsymbol{v}_n^k$ is known from the previous iteration. This enables the parallelization of the fine propagator, since only the coarse terms $\Phi_G\boldsymbol{v}_n^{k+1}$ have to be computed sequentially at each iteration.

The algorithm is proven to converge in a finite number of iterations to the fine approximation (see \cite{finiteconv2018gander} for further details). However, note that it is worthless to run the Parareal algorithm until the fine approximation is obtained, since the sequential counterpart is notoriously cheaper than the required number of iterations of the algorithm. Instead, it is suggested to run the method until a similar order of magnitude as for the error given by the target solution is obtained (cf. \cite{convcriteria2015arteaga}).

As the Parareal algorithm is essentially an iterative method, it is straightforward to compute its iteration matrix. A single iteration of the algorithm for all time steps can be rewritten in matrix form as
\begin{equation*}
    M_G \boldsymbol{V}^{k+1} = (M_G - M_F)\boldsymbol{V}^k + \boldsymbol{b}_F.
\end{equation*}
If we further define the error between the approximation of the Parareal algorithm at the $k$-th iteration and the fine approximation as $\boldsymbol{e}^{k} = \boldsymbol{V}^k - \boldsymbol{V}^F$, following \cite{itermatrix2021buvoli} and \cite{itermatrix2018ruprecht}, it can be shown that the recursive relation $\boldsymbol{e}^{k+1} = (I - M_G^{-1}M_F)\boldsymbol{e}^k$ is satisfied. In this framework, the iteration matrix $E = I - M_G^{-1}M_F$ is usually considered in order to analyze the convergence properties of the Parareal algorithm (see \cite{itermatrix2021buvoli} and \cite{itermatrix2018ruprecht} for further details).

\subsection{Convergence conditions}

As mentioned in the introduction, the convergence of the Parareal algorithm is by no means straightforward to analyze. Gander and Vandewalle (cf. \cite{scalaranalysis2007gander}) first attempted to prove convergence for the scalar Dahlquist test problem, being able to show significant improvement of the method if it was implemented for parabolic problems. As explained above, further extensions of such an analysis have included the convergence analysis of the system \cref{eq:ivp}, under demanding conditions for the matrix $A$. In \cite{convergence2024shulin}, this approach is considered in order to reduce the matrix problem into several scalar problems, taking advantage of a Schur decomposition. To the best of our knowledge, very few papers discuss the convergence properties of the Parareal algorithm for more general cases, such as problem \cref{eq:biot_system} with $A$ not necessarily symmetric, and even fewer succeed to address theoretical convergence conditions in this setting. The aim of this work is to further extend these conditions for a more general framework, such that they are suitable for problem \cref{eq:biot_system}.

As mentioned previously, the starting point of the analysis is the expression $\boldsymbol{e}^{k+1} = (I - M_G^{-1}M_F)\boldsymbol{e}^k$, or equivalently, $M_G\boldsymbol{e}^{k+1} = (M_G - M_F)\boldsymbol{e}^k$. Most approaches attempt to address the convergence with respect to the iteration matrix $E$, trying to show that its norm can be bounded by one. In this proposal, we consider the approach given by \cite{matrixapproach2024nuca}, assuming coercivity of $M_G$ and continuity of $M_G-M_F$.

\begin{proposition} \label{Prop_MG}
    Let us assume that operator $M_G$ given by \cref{eq:notation_matrix2} is coercive, with coercivity constant $\alpha_{M_G}$, and that the inequality $\|M_G-M_F\| < \alpha_{M_G}$ holds. Then, the Parareal algorithm is convergent.
\end{proposition}
\begin{proof}
    The expression $M_G\boldsymbol{e}^{k+1} = (M_G - M_F)\boldsymbol{e}^k$ is multiplied by the vector $\boldsymbol{e}^{k+1}$, obtaining
    \begin{displaymath}
        \langle M_G \boldsymbol{e}^{k+1}, \boldsymbol{e}^{k+1} \rangle = \langle (M_G - M_F)\boldsymbol{e}^k, \boldsymbol{e}^{k+1}\rangle.
    \end{displaymath}
    By definition of coercivity, it follows $\langle M_G\boldsymbol{e}^{k+1},\boldsymbol{e}^{k+1}\rangle \geq \alpha_{M_G} \|\boldsymbol{e}^{k+1}\|^2$. On the other hand, by applying Cauchy-Schwarz inequality, we obtain
    \begin{align*}
        \alpha_{M_G}\|\boldsymbol{e}^{k+1}\|^2 & \leq \langle (M_G - M_F)\boldsymbol{e}^k, \boldsymbol{e}^{k+1} \rangle \\ & \leq \|M_G - M_F\| \|\boldsymbol{e}^k\|\|\boldsymbol{e}^{k+1}\|,
    \end{align*}
    Thus, it follows the inequality
    \begin{equation*}
        \|\boldsymbol{e}^{k+1}\| \leq \frac{\|M_G-M_F\|}{\alpha_{M_G}}\|\boldsymbol{e}^k\|,
    \end{equation*} 
    and the Parareal algorithm guarantees convergence for every iteration.
\end{proof}

Therefore, it is sufficient to show that the operator $M_G$ is coercive and that $M_G$ and $M_F$ are sufficiently close to each other so as to prove convergence of the Parareal algorithm. However, we would sooner work directly on the operator $\Phi_G$, rather than proving coercivity of $M_G$. For that purpose, we state a sufficient condition on $\Phi_G$ such that the property is satisfied for $M_G$.

\begin{proposition} \label{prop:contractionv2}
    Let $\|\cdot\|$ be any vector norm induced by an inner product. Then, if $\|\Phi_G\| < 1$ is satisfied for its induced matrix norm, the operator $M_G$ defined in \cref{eq:notation_matrix2} is coercive.
\end{proposition}
\begin{proof}
    Let us define the arbitrary vector $\boldsymbol{V} = \begin{pmatrix}
        \boldsymbol{v}_0^\top & \boldsymbol{v}_1^\top & \cdots & \boldsymbol{v}_{N_c}^\top
    \end{pmatrix}^\top$. Then, the operator $M_G$ being coercive is equivalent to satisfying the inequality $\langle M_G \boldsymbol{V},\boldsymbol{V}\rangle \ge \alpha_{M_G}\|\boldsymbol{V}\|^2$, with $\alpha_{M_G} > 0$. The case $\boldsymbol{V} = \boldsymbol{0}$ is straightforward to prove. For $\boldsymbol{V} \neq \boldsymbol{0}$, expanding the previous product, it follows
    \begin{equation} \label{eq:prop2_form2}
        \langle M_G \boldsymbol{V}, \boldsymbol{V}\rangle = \sum_{i=0}^{N_c} \langle \boldsymbol{v}_i, \boldsymbol{v}_i \rangle - \sum_{i=0}^{N_c-1}\boldsymbol \langle \Phi_G \boldsymbol{v}_i,\boldsymbol{v}_{i+1} \rangle = \sum_{i=0}^{N_c} \| \boldsymbol{v}_i\|^2 - \sum_{i=0}^{N_c-1}\boldsymbol \langle \Phi_G \boldsymbol{v}_i,\boldsymbol{v}_{i+1} \rangle.
    \end{equation}

    On the other hand, let us consider the following norm:
    \begin{equation} \label{eq:prop2_form1}
        \left\|\lambda_i \boldsymbol{v}_{i+1} - \frac{1}{2\lambda_i}\Phi_G\boldsymbol{v}_i\right\|^2 = \lambda_i^2 \| \boldsymbol{v}_{i+1} \|^2 + \frac{1}{4\lambda_i^2}\|\Phi_G\boldsymbol{v}_i\|^2 - \langle \Phi_G\boldsymbol{v}_i, \boldsymbol{v}_{i+1}\rangle,
    \end{equation}
    where $\lambda_i > 0$ is a parameter to be determined. Observe that this equality only holds for vector norms induced by an inner product, which yields directly from the parallelogram law. Then, by summation of the terms \cref{eq:prop2_form1} for $i\in\{0,1,\dots, N_c-1\}$, we obtain
    \begin{align*}
        \sum_{i=0}^{N_c-1} \left\| \lambda_i \boldsymbol{v}_{i+1} - \frac{1}{2\lambda_i}\Phi_G\boldsymbol{v}_i \right\|^2 = &~ \frac{1}{4\lambda_0^2}\|\Phi_G\boldsymbol{v}_0\|^2 + \sum_{i=1}^{N_c-1} \lambda_{i-1}^2\|\boldsymbol{v}_i\|^2 + \sum_{i=1}^{N_c-1} \frac{1}{4\lambda_i^2}\|\Phi_G\boldsymbol{v}_i\|^2 \\
        &~ + \lambda_{N_c-1}^2\|\boldsymbol{v}_{N_c}\|^2 - \sum_{i=0}^{N_c-1}\langle\Phi_G\boldsymbol{v}_i, \boldsymbol{v}_{i+1}\rangle.
    \end{align*}

    Next, we consider the subtraction between \cref{eq:prop2_form2} and the previous expression, yielding
    \begin{align*}
        \langle M_G \boldsymbol{V}, \boldsymbol{V}\rangle &~- \sum_{i=0}^{N_c-1} \left\| \lambda_i \boldsymbol{v}_{i+1} - \frac{1}{2\lambda_i}\Phi_G\boldsymbol{v}_i \right\|^2 = 
        \|\boldsymbol{v}_0\|^2 - \frac{1}{4\lambda_0^2} \|\Phi_G\boldsymbol{v}_0\|^2  \\
        &~ + \sum_{i=1}^{N_c-1} (1 - \lambda_{i-1}^2)\|\boldsymbol{v}_i\|^2 - \sum_{i=1}^{N_c-1} \frac{1}{4\lambda_i^2} \|\Phi_G\boldsymbol{v}_i\|^2 + (1-\lambda_{N_c-1}^2)\|\boldsymbol{v}_{N_c}\|^2.
    \end{align*}
    Observe that if the subtraction is positive, $M_G$ is guaranteed to be positive definite. Let us consider $\lambda_i = 1/\sqrt{2}$ for all $i \in \{0,1,\dots,N_c-1\}$. Then, we obtain
    \begin{align*}
        \langle M_G \boldsymbol{V}, \boldsymbol{V}\rangle &~- \sum_{i=0}^{N_c-1} \left\| \frac{1}{\sqrt{2}} \boldsymbol{v}_{i+1} - \frac{1}{\sqrt{2}}\Phi_G\boldsymbol{v}_i \right\|^2 = 
        \|\boldsymbol{v}_0\|^2 - \frac{1}{2} \|\Phi_G\boldsymbol{v}_0\|^2  \\
        &~ + \frac{1}{2}\sum_{i=1}^{N_c-1} \left(\|\boldsymbol{v}_i\|^2 -  \|\Phi_G\boldsymbol{v}_i\|^2\right) + \frac{1}{2}\|\boldsymbol{v}_{N_c}\|^2.
    \end{align*}
    Then, since $\|\Phi_G\| < 1$ is fulfilled, it follows
    \begin{equation*}
        \|\boldsymbol{v}_i\| > \|\Phi_G\|\|\boldsymbol{v}_i\| \geq \|\Phi_G\boldsymbol{v}_i\|,
    \end{equation*}
    or equivalently, $\|\boldsymbol{v}_i\|-\|\Phi_G\boldsymbol{v}_i\| > 0$. Since we have assumed that $\boldsymbol{V} \neq \boldsymbol{0}$, at least one term $\boldsymbol{v}_i$ must contain a nonzero entry. If that entry is located in the vector $\boldsymbol{v}_0$, $\|\boldsymbol{v}_0\|^2 - \frac{1}{2}\|\Phi_G\boldsymbol{v}_0\|^2 > 0$ holds. Similarly, if the nonzero entry is located in $\boldsymbol{v}_i$, for any $i \in \{1,2,\dots,N_c-1\}$, it follows $\|\boldsymbol{v}_i\|^2 - \|\Phi_G\boldsymbol{v}_i\|^2 > 0$. Finally, $\|\boldsymbol{v}_{N_c}\|^2 > 0$ holds if the vector $\boldsymbol{v}_{N_c}$ contains a nonzero entry. In all three cases,
    \begin{equation*}
        \langle M_G \boldsymbol{V}, \boldsymbol{V}\rangle - \sum_{i=0}^{N_c-1} \left\| \frac{1}{\sqrt{2}} \boldsymbol{v}_{i+1} - \frac{1}{\sqrt{2}}\Phi_G\boldsymbol{v}_i \right\|^2 > 0,
    \end{equation*}
    and it can be concluded that $\langle M_G\boldsymbol{V},\boldsymbol{V}\rangle > 0$.
    
    Finally, let us define the compact unit sphere $S$ of same dimension as vector $\boldsymbol{V}$. Then, it is well known that there exists a minimum such that $\alpha_{M_G}:=\min_{\boldsymbol{V}\in \partial S} \langle M_G \boldsymbol{V}, \boldsymbol{V} \rangle$. Therefore, since $\boldsymbol{V} \neq \boldsymbol{0}$, let us define $\boldsymbol{U} = \frac{1}{\|\boldsymbol{V}\|}\boldsymbol{V}$, which satisfies $\boldsymbol{U} \in S$. Then, it follows $\langle M_G \boldsymbol{V}, \boldsymbol{V}\rangle = \langle M_G \|\boldsymbol{V}\|\boldsymbol{U}, \|\boldsymbol{V}\|\boldsymbol{U}\rangle = \|V\|^2 \langle M_G \boldsymbol{U}, \boldsymbol{U}\rangle \geq \alpha_{M_G}\|\boldsymbol{V}\|^2$, which proves the claim.
\end{proof}

Consequently, a new condition for the analysis of the convergence of the Parareal algorithm is obtained, dependent only on $\Phi_G$. Remarkably, this condition is closely related to the outperformance of the Parareal algorithm for parabolic problems with clear dissipative behaviour. In the following, it will thus be sufficient to prove that the coarse numerical approximation preserves the contractive property, together with a fine approximation sufficiently close to the coarse one, in order to guarantee the convergence of the Parareal algorithm.

\section{Extension of the Biot's model and iterative solvers} \label{sec:biot}

This section aims to describe the reference problem that inspired the implementation of the Parareal algorithm for coupled elliptic-parabolic problems, that is, the quasi-static Biot model. Once discretized in space, a more general set of differential equations can be obtained, which is embedded into the family of equations given by \cref{eq:biot_system}. For those equations, we derive an extended formulation of the fixed-stress and multirate fixed-stress method, originally proposed for Biot's model.

\subsection{Quasi-static Biot model}\label{subsec2.1}

We consider here the quasi-static Biot model \cite{Biot1956}, a well-established model for flow and deformation in porous media.  Let $\Omega$ be a bounded Lipschitz domain
in $\mathbb{R}^{d}$, $d=2,3$. We are looking to solve the following two equations for the unknowns $u$, which stays for the displacement, and $p$, the fluid pressure:

\begin{subequations}\label{eq:nlBiot}
\begin{align}
-\nabla \cdot \sigma + \alpha \nabla p &= f~~ \text{in}~~ \Omega\times (0,T),\label{nlBiot1}\\
\frac{\partial}{\partial t}\left( \frac{1}{M}p+\alpha \nabla \cdot u\right) - \nabla \cdot (K \nabla p)  &=g\;\;\text{in}~~ \Omega\times (0,T).
\label{nlBiot2}
\end{align}
\end{subequations}
The first equation above is a momentum balance, while the second stays for the mass balance. Observe that the equations are fully coupled. Moreover, $\sigma=2 \mu \varepsilon(u) + \lambda \nabla \cdot u I$ denotes the effective stress, where
$\varepsilon(u) := \frac{1}{2} (\nabla u + (\nabla u)^T)$ is the linearized strain tensor. Furthermore, $\lambda$ and $\mu$ are the Lam\'{e} parameters, $\alpha$ is the Biot-Willis coefficient, $M$ is the Biot modulus, and $K$ is the hydraulic conductivity tensor.

In the following, we assume that the poro-elastic medium is heterogenous, isotropic and saturated, whereas the fluid is slightly compressible. In turn, all the physical parameters are assumed to be positive, although they may vary in space and time. The permeability tensor $K$ is assumed to be symmetric, bounded, uniformely positive definite in space and constant in time.

With proper boundary and initial conditions, the system~\cref{eq:nlBiot} has a unique solution, see, for instance, \cite{Showalter2000,  Zenisek1984}. For the sake of simplicity, we consider the following homogeneous boundary conditions:
\begin{align*}
    p &= 0               && \text{on } \Gamma_{p,D}, &
    (K \nabla p)\cdot n &= 0 && \text{on } \Gamma_{p,N}, \\
    u &= 0               && \text{on } \Gamma_{u,D}, &
    \sigma n            &= 0 && \text{on } \Gamma_{u,N},
\end{align*}
where 
$\Gamma_{p,D} \cap \Gamma_{p,N} = \emptyset$,
$\overline{\Gamma}_{p,D}\cup \overline{\Gamma}_{p,N}=\Gamma=\partial{\Omega}$,
$\Gamma_{u,D} \cap \Gamma_{u,N} = \emptyset$ and
$\overline{\Gamma}_{u,D} \cup \overline{\Gamma}_{u,N}=\Gamma$
are fulfilled. Nevertheless, the presented solvers and their analysis carry over to other scenarios with minor modifications. Finally, we assume $p(x,0) = p_0(x)$ and $u(x,0) = u_0(x)$ initial conditions compatible with the imposed boundary conditions.

Then, the variational form of the equation \cref{eq:nlBiot} reads: find $(u,p) \in \mathcal{C}^1([0,T]; V) \times \mathcal{C}^1([0,T]; P)$ such that
\begin{subequations}
\begin{align*}
2\mu(\varepsilon(u),\varepsilon(v)) + \lambda (\nabla \cdot u, \nabla \cdot v) - \alpha(p, \nabla \cdot v) &= (f,v) , \qquad \forall v \in V,\\
\frac{1}{M}\left(\frac{\partial p}{\partial t},q\right) + \alpha\left(\frac{\partial}{\partial t}\left(\nabla \cdot u\right), q\right) + (K  \nabla p, \nabla q) &= (g, q), \qquad \forall q \in P,
\end{align*}
\end{subequations}
where the spaces are defined as $V = \{v \in (H^1(\Omega))^d~:~ v|_{\Gamma_{u,D}} = 0\}$ and $P = \{q \in H^1(\Omega)~:~ q|_{\Gamma_{p,D}} = 0\}$.

Let us briefly comment on possible discretizations for the spaces $V$ and $P$. A remarkable property that the pair of finite element discrete spaces have to satisfy is the inf-sup condition. A well-known choice is the stabilized P1-P1 scheme, consisting of the space of piecewise linear continuous vector-valued functions and the space of piecewise linear continuous scalar-valued functions for $V$ and $P$, respectively. This approach was first introduced in \cite{Aguilar_et_al_2008} and further analysed in \cite{Rodrigo_et_al_2016}. This last article also proposes using MINI elements for displacement, adding a space of bubble functions to the previous discrete space. Finally, another choice would be P2-P1, which includes a space of piecewise quadratic continiuous vector-valued functions for $V$, as studied in \cite{MuradLoula1992, MuradThomeeLoula1996}.

If the pair of discrete spaces is denoted by $V_h \times P_h$, the discrete variational form of the equation \cref{eq:nlBiot} reads: find $(u_h,p_h) \in \mathcal{C}^1([0,T]; V_h) \times \mathcal{C}^1([0,T]; P_h)$ such that
\begin{subequations}\label{eq:nlBiot_weak}
\begin{align}
2\mu(\varepsilon(u_h),\varepsilon(v_h)) + \lambda (\nabla \cdot u_h, \nabla \cdot v_h) - \alpha(p_h, \nabla \cdot v_h) &= (f,v_h) , \qquad \forall v_h \in V_h, \label{nlBiot_weak1}\\
\frac{1}{M}\left(\frac{\partial p_h}{\partial t},q_h\right) + \alpha\left(\frac{\partial}{\partial t}\left(\nabla \cdot u_h\right), q_h\right) + (K  \nabla p_h, \nabla q_h) &= (g, q_h), \qquad \forall q_h \in P_h.  \label{nlBiot_weak2}
\end{align}
\end{subequations}
Observe that the notation $u_h \equiv u_h(t)$ and $p_h \equiv p_h(t)$ is adopted for the sake of convenience.

Finally, note that this problem can be included in the family of the coupled elliptic-parabolic problems \cref{eq:biot_system} with $A_{uu}$ and $A_{pp}$ symmetric and positive-definitive, $A_{pu} = 0$, $C_{pu} = -A_{up}^\top$ and $C_{pp} = c_fI$, with $c_f \in \mathbb{R}_+$. Therefore, we henceforth study the coupled problem given by
\begin{equation} \label{eq:biot_discr_space}
    \begin{pmatrix}
        0 & 0 \\ A_{up}^\top & c_fI
    \end{pmatrix} \frac{d}{dt}\begin{pmatrix}
        u \\ p
    \end{pmatrix} + \begin{pmatrix}
        A_{uu} & -A_{up} \\ 0 & A_{pp}
    \end{pmatrix} \begin{pmatrix}
        u \\ p
    \end{pmatrix} = \begin{pmatrix}
        f \\ g
    \end{pmatrix}.
\end{equation}
Observe that the notation $u \equiv u_h(t)$ and $p \equiv p_h(t)$ is employed hereafter for simplicity.

\begin{remark}
    Other coupled elliptic-parabolic problems may not be structured as \cref{eq:biot_discr_space}, including, for instance, the electrochemical systems described in \cite{saralee2016thesis}. Such problems are out of the scope of this work, but analogous procedure can lead to results for the convergence of the Parareal algorithm applied to them.
\end{remark}

\subsection{Iterative solvers for the coupled problem} \label{subsec:iterative}

This subsection is devoted to the description of the iterative solvers which stem from the schemes that are typically employed for the numerical integration of the quasi-static Biot model \cref{eq:nlBiot_weak}, namely, the fixed-stress method and the multirate fixed-stress iterative scheme. In particular, these two methods are generalized for their implementation as integrators of problem \cref{eq:biot_discr_space}.

First, on a given equidistant partition $0 = t_0 < t_1 < \cdots < t_{N-1} < t_N = T$ of the the interval $[0,T]$ into subintervals $[t_{n-1},t_n]$, of length $\tau = t_n - t_{n-1}$, $n = 1,2,\dots,n$, we can discretize in time the problem \cref{eq:biot_discr_space} by the backward Euler method. Let us further denote the unknowns at time $t_n$ by $u_n:=u(t_n)$ and $p_n:=p(t_n)$. We use the short notation $f_n = f(t_n)$ and $\tilde{g}_n = \tau g(t_n) + A_{up}^\top u_{n-1} + c_fp_{n-1}$ for the right-hand sides in the backward Euler time-step equations.

Then, the fully-discrete form of the problem \cref{eq:biot_discr_space} reads: given $u_{n-1}$ and $p_{n-1}$, find $u_n$ and $p_n$ such that there holds:
\begin{subequations} \label{eq:fully_discrete_problem}
    \begin{align} 
	A_{uu}u_{n+1} - A_{up}p_{n+1} & = f_{n+1}, \label{eq:elliptic} \\ 
	A_{up}^\top u_{n+1} + (c_fI +A_{pp})p_{n+1} & = \tilde{g}_{n+1}, \label{eq:parabolic}
\end{align}
\end{subequations}
for all $n = 1,2,\dots,N$. If solved directly, we would follow the so-called monolithic approach. However, the direct solver might not be the most efficient strategy on the grounds of the large size of the resulting system of linear equations.

Another alternative lies in the implementation of stabilized iterative splitting methods. In the context of poroelasticity, the most popular iterative methods are the fixed-stress split and the undrained split \cite{kim_undrained2011}. Both methods involve a tuning parameter $L \in \mathbb{R}_+$ connected to the added stabilization term. In this work, we will use the approach provided by the fixed-stress method, although similar schemes can be derived by using the undrained split.

In the following, we will use the index $i$ for the iterations. At each time step $t = t_{n+1}$ and iteration $i \geq 1$, $u_{n+1}^i$ and $p_{n+1}^i$ will be an approximation of the solution $u(t_{n+1})$ and $p(t_{n+1})$ of the system \cref{eq:fully_discrete_problem}, respectively. The iterations are typically starting with the solution at the last time step, that is, $u^0_{n+1} = u_{n}$ and $p_{n+1}^0 = p_{n}$.

The stabilized iterative splitting (hereafter abbreviated as FS) method is described in \Cref{alg1_new}. One starts by solving the equation arising from the parabolic equation for $p_{n+1}^{i}$, using $u^{i-1}_{n+1}$ from the previous iteration. Note that a stabilization term $L(p_{n+1}^i - p_{n+1}^{i-1})$ is added, where $L \in \mathbb{R}_+$ is a free to be chosen tuning parameter, necessary to ensure the convergence of the scheme. After obtaining $p^i_{n+1}$, one solves the equation which follows from the discretization of the elliptic equation, in order to find $u^i_{n+1}$. One iterates until a certain convergence criterion is fulfilled.

\begin{algorithm}
\caption{FS method in fully discrete form}
\label{alg1_new}
\begin{algorithmic}[1]
    \FOR{$n = 0,1,2,\dots,N-1$}
    \STATE{Set $u_{n+1}^0 = u_{n}$ and $p_{n+1}^0 = p_{n}$.}
    \FOR{$i = 1,2,\dots$ until convergence}
    \STATE{Given $u_{n+1}^{i-1}$ and $p_{n+1}^{i-1}$, find $p_{n+1}^i$ such that there holds $$(c_fI + \tau A_{pp})p_{n+1}^i + Lp_{n+1}^i = Lp_{n+1}^{i-1} - A_{up}^\top u_{n+1}^{i-1} + \tilde{g}_{n+1}.$$}
    \STATE{Given $p_{n+1}^i$, find $u_{n+1}^i$ such that there holds $$A_{uu} u_{n+1}^i = A_{up} p_{n+1}^i + f_{n+1}.$$}
    \ENDFOR
    \ENDFOR
\end{algorithmic}
\end{algorithm}

\begin{remark}
    The fixed-stress splitting scheme, which happens to be a particular case of the FS method, has been extensively studied for the quasi-static Biot problem and extensions. We refer to \cite{Storvik2019} for a first convergence analysis, then to \cite{Both2017} for an extension to heterogeneous media. In the latter paper, the convergence was obtained in terms of energy norms.
\end{remark}

\begin{remark}
    The speed of convergence of the FS method depends on the choice of the tuning parameter $L$. We refer to \cite{Rooseetal2003} for a discussion on the choice of $L$. A particular algorithm for choosing an ``optimal'' $L$ is proposed there as well. We mention that an alternative way to enhance the speed of convergence lies in the use of the Anderson acceleration. In \cite{both2019anderson} this was investigated for an extension of the Biot model to variable saturated flow.
\end{remark}

On the other hand, the multirate scheme for the coupled elliptic-parabolic problems introduces a temporal splitting strategy that decouples the elliptic and parabolic subproblems, allowing to take different time steps. This is combined with a regular splitting strategy (such as the FS method), yielding a multirate method that employs larger time steps for the elliptic equation and smaller steps for the parabolic equation.

The physical phenomenon which served as inspiration is the slower mechanical deformation, in contrast to the faster pressure equation arising from the Biot problem. A first attempt to introduce the multirate strategy to porous media and flow problems was made by \cite{girault2016convergence, mikelic2013} for the Biot model and by \cite{ShanZhengLayton2013, xiong2012report} for the evolutive Stokes-Darcy model. However, in this work, we consider the approach given by \cite{almani2016convergence}, while generalizing it for the problem \cref{eq:fully_discrete_problem}.

As already mentioned, the multirate method requires two different time scales for the numerical integration of both equations (the elliptic and the parabolic). In particular, the same time discretization with step size $\tau$ will be employed for the integration of the parabolic equation, whereas a $q$ times coarser time step is considered for the elliptic equation. In the following, let us denote by $n$ the coupling iteration index, by $i$ the time step iteration index of the elliptic equation, and by $j$ the fine time step iteration index of the parabolic equation for each step within the elliptic equation. Then, the multirate scheme based on the FS method is given by \Cref{alg_2_new}.

\begin{algorithm}
    \caption{Multirate Iterative Coupling Algorithm}
    \label{alg_2_new}
    \begin{algorithmic}[1]
    \FOR{$n = 0,q,2q,\dots,N-q$}
    \STATE{Set $u_{n+q}^0 = u_n$ and $p_{n+j}^0 = p_n$ for all $j \in \{1,2,\dots,q\}$.}
    \FOR{$i = 1,2,\dots$ until convergence}
    \STATE{Set $p_n^i = p_n$.}
    \FOR{$j = 1,2,\dots,q$}
    \STATE{Given $u_n$, $u_{n+q}^{i-1}$, $p_{n+j-1}^{i-1}$, $p_{n+j}^{i-1}$, and $p_{n+j-1}^i$, find $p_{n+j}^i$ satisfying \begin{equation*} \begin{split}
        (c_f&+L)\left(\frac{p_{n+j}^i-p_{n+j-1}^i}{\tau}\right) + A_{pp} p_{n+j}^i \\ &= L\left(\frac{p_{n+j}^{i-1}-p_{n+j-1}^{i-1}}{\tau}\right) - A_{up}^\top \left(\frac{u_{n+q}^{i-1}-u_n}{q\tau}\right) + g(t_{n+j}). \end{split}
    \end{equation*}}
    \ENDFOR
    \STATE{Given $p_{n+q}^i$, find $u_{n+q}^i$ satisfying $$A_{uu}u_{n+q}^i = A_{up}p_{n+q}^i + f_{n+q}.$$}
    \ENDFOR
    \ENDFOR
    \end{algorithmic}
\end{algorithm}

The main advantage of the multirate scheme lies in avoiding the computation of the extra time steps for the elliptic problem, which, more often than not, is not strictly required for achieving an accurate approximation.

\begin{remark}
    The speed of convergence is once again dependent on the choice of the parameter $L$. However, convergence is guaranteed for a suitable $L$, as it is shown in \cite{almani2016convergence} for the Biot model. Further extensions of the multirate scheme, as well as their corresponding results, can be found in \cite{Almani2024,almani2023convergence}.
\end{remark}

\begin{remark}
    As shown in \cite{almani2016convergence}, the multirate FS method does not converge to the solution given by \cref{eq:fully_discrete_problem}. Instead, the numerical solution towards which the scheme converges is given by the following set of equalities:
    \begin{subequations} \label{eq:fully_discrete_problem_multirare}
    \begin{align} 
	A_{uu}u_{n+q} - A_{up}p_{n+q} & = f_{n+q}, \label{eq:elliptic_multirate} \\ 
	A_{up}^\top \left(\frac{u_{n+q}-u_n}{q\tau} \right) + c_f\frac{p_{n+j}-p_{n+j-1}}{\tau} + A_{pp}p_{n+j} & = g(t_{n+j}), \quad j \in \{1,\dots,q\}.\label{eq:parabolic_multirate}
\end{align}
\end{subequations}
    This numerical solution is also shown to be close to the solution given by the monolithic implementation (see the numerical experiments provided by \cite{almani2016convergence}).
\end{remark}

All in all, both FS and multirate FS methods are employed as the coupling iterative methods for solving problem \eqref{eq:fully_discrete_problem}. The diagrams in \Cref{fig:FS_flowchart} and \Cref{fig:MFS_flowchart} show a clear comparison between both schemes. In this work, the methods are considered as the coarse and fine propagators of the Parareal algorithm, and using the results obtained in \Cref{sec:parareal}, we will determine conditions over the coarse time step in order to ensure the convergence of the algorithm as a whole.

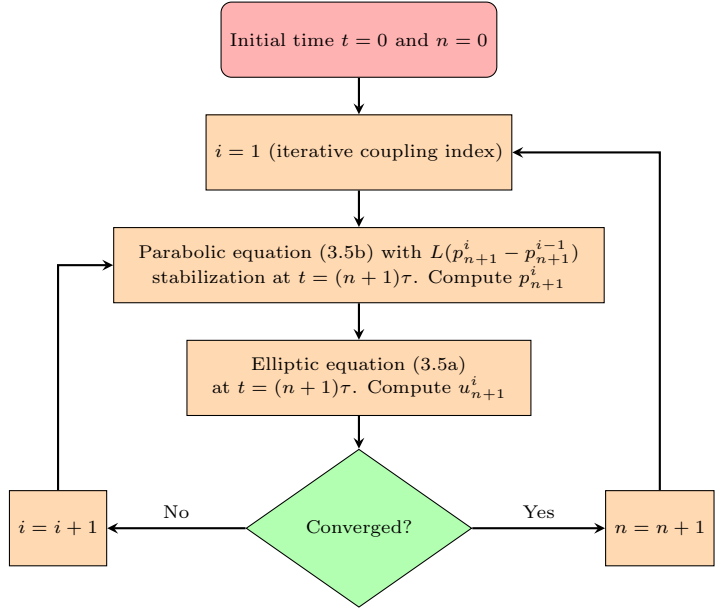
\begin{figure}
    \centering
    \begin{tikzpicture}[node distance=1.5cm]
        \node (start) [startstop] {\scriptsize Initial time $t = 0$ and $n = 0$};
        \node (process1) [process, below of=start] {\scriptsize $i = 1$ (iterative coupling index)};
        \node (process2) [process, below of=process1] {\scriptsize \begin{tabular}{c}
            Parabolic equation \cref{eq:parabolic} with
            $L(p_{n+1}^i - p_{n+1}^{i-1})$ \\ stabilization at $t = (n+1)\tau$. Compute $p_{n+1}^i$
        \end{tabular}};
        \node (process3) [process, below of=process2] {\scriptsize \begin{tabular}{c}
            Elliptic equation \cref{eq:elliptic} \\ at $t = (n+1)\tau$. Compute $u_{n+1}^i$
        \end{tabular}};
        \node (decision1) [decision, below of=process3, yshift = -0.5cm] {\scriptsize Converged?};
        \node (process4) [process, left of=decision1, xshift = -2.5cm] {\scriptsize $i = i+1$};
        \node (process5) [process, right of=decision1, xshift = 2.5cm] {\scriptsize $n = n+1$};
    
        \draw [arrow] (start) -- (process1);
        \draw [arrow] (process1) -- (process2);
        \draw [arrow] (process2) -- (process3);
        \draw [arrow] (process3) -- (decision1);
        \draw [arrow] (decision1) -- node[anchor=south] {\scriptsize No} (process4);
        \draw [arrow] (decision1) -- node[anchor=south] {\scriptsize Yes} (process5);
        \draw [arrow] (process4) |- (process2);
        \draw [arrow] (process5) |- (process1);
    \end{tikzpicture}
    \caption{Flowchart for the FS method described in \Cref{alg1_new}.}
    \label{fig:FS_flowchart}
\end{figure}

\begin{figure}
    \centering
    \begin{tikzpicture}[node distance=1.5cm]
        \node (start) [startstop] {\scriptsize Initial time $t = 0$ and $n = 0$ };
        \node (process1) [process, below of=start] {\scriptsize $i = 1$ (iterative coupling index)};
        \node (process6) [process, below of=process1] {\scriptsize $j = 1$ (parabolic iteration index)};
        \node (process2) [process, below of=process6] {\scriptsize \begin{tabular}{c}
            Parabolic equation \cref{eq:parabolic} with
            $L(p_{n+1}^i - p_{n+1}^{i-1})$ \\ stabilization at $t = (n+1)\tau$. Compute $p_{n+1}^i$
        \end{tabular}};
        \node (decision2) [decision, below of=process2] {\scriptsize $j = q$?};
        \node (process3) [process, below of=decision2, yshift=-0.5cm] {\scriptsize \begin{tabular}{c}
            Elliptic equation \cref{eq:elliptic} \\ at $t = (n+1)\tau$. Compute $u_{n+1}^i$
        \end{tabular}};
        \node (decision1) [decision, below of=process3, yshift = -0.5cm] {\scriptsize Converged?};
        \node (process4) [process, left of=decision1, xshift = -3.5cm] {\scriptsize $i = i+1$};
        \node (process5) [process, right of=decision1, xshift = 3.5cm] {\scriptsize $n = n+1$};
        \node (process7) [process, left of=decision2, xshift = -2cm] {\scriptsize $j = j+1$};
    
        \draw [arrow] (start) -- (process1);
        \draw [arrow] (process1) -- (process6);
        \draw [arrow] (process6) -- (process2);
        \draw [arrow] (process2) -- (decision2);
        \draw [arrow] (decision2) -- node[anchor=west] {\scriptsize Yes} (process3);
        \draw [arrow] (decision2) -- node[anchor=south] {\scriptsize No} (process7);
        \draw [arrow] (process3) -- (decision1);
        \draw [arrow] (decision1) -- node[anchor=south] {\scriptsize No} (process4);
        \draw [arrow] (decision1) -- node[anchor=south] {\scriptsize Yes} (process5);
        \draw [arrow] (process4) |- (process6);
        \draw [arrow] (process5) |- (process1);
        \draw [arrow] (process7) |- (process2);
    \end{tikzpicture}
    \caption{Flowchart for the multirate FS method described in \Cref{alg_2_new}.}
    \label{fig:MFS_flowchart}
\end{figure}
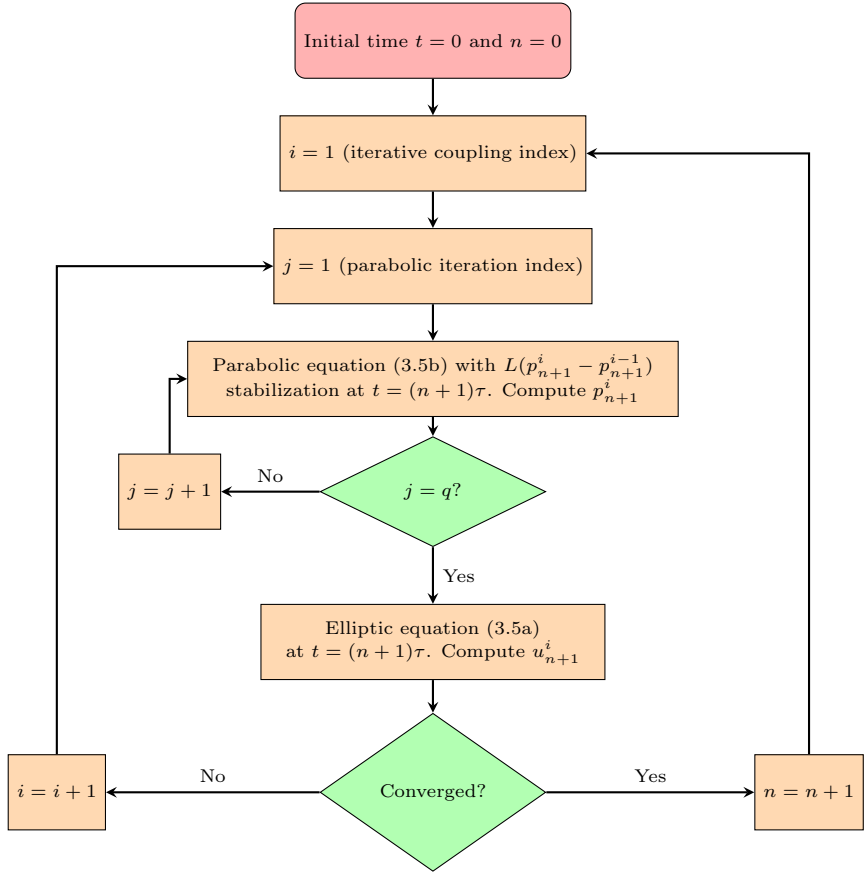

\section{Convergence of Parareal for the coupled problem} \label{sec:contraction}

This section contains the conditions (if required) under which the Parareal algorithm is guaranteed to converge for the problem \cref{eq:biot_system}, provided a suitable vector norm. For the sake of simplicity, no source term will be considered. Recalling the main results that ensure convergence for the Parareal algorithm, stated in \Cref{Prop_MG} and \Cref{prop:contractionv2}, the conditions $\|\Phi_G\|<1$ and $\|M_G-M_F\| < \alpha_{M_G}$ should be guaranteed, where $\alpha_{M_G}$ stands for the coercivity coefficient of $M_G$. The latter condition is generally satisfied for sufficiently close coarse and fine time steps. However, the former condition may be more restrictive. Thus, this section includes the constrains over $\Delta T$ such that $\|\Phi_G\| < 1$ holds. This is proven for the three solvers described in \Cref{subsec:iterative} for the problem \eqref{eq:biot_system}, namely, the monolithic, FS and multirate FS methods.

Before beginning with the proposed properties, let us denote by $\|\cdot\|_{a,b}$ the weighted norm defined by
\begin{equation} \label{eq:weighted_norm}
    \|(u,p)\|^2_{a,b} = a\|u\|^2_2 + b\|p\|^2_2,
\end{equation}
where $a,b \in \mathbb{R}_+$ and $\|\cdot\|_2$ denotes the discrete $\ell_2$-norm. Similarly, we define the discrete energy norm denoted by $\|\cdot\|_e$ as
\begin{equation*}
    \|(u,p)\|_e^2 = (A_{uu}u,u) + c_f\|p\|_2^2 = \|u\|_{A_{uu}}^2 + c_f\|p\|_2^2,
\end{equation*}
where $\|\cdot \|_{A_{uu}}$ represents the induced norm by the scalar product $(A_{uu}\cdot,\cdot)$.

Furthermore, we recall the following result, which is employed later on in this section.

\begin{lemma} \label{lemma}
        For a set of vectors $\{x_j\}_{j\in\{1,\dots,n\}}$ and a vector norm induced by an inner product denoted by $\|\cdot\|$, the following inequality is satisfied:
        \begin{equation*}
            \|x_1 + \cdots + x_n\|^2 \leq n(\|x_1\|^2 + \cdots + \|x_n\|^2).
        \end{equation*}
    \end{lemma}
    \begin{proof} 
        The result follows immediately from the triangle and the Cauchy-Schwarz inequalities.
    \end{proof}

Finally, we remark that the initial condition of the problem must be nonzero for both $p$ and $u$ unknowns. Otherwise, the system of equations would admit only the trivial solution, since source terms are assumed to be zero, making the point of the contractivity analysis meaningless.

\subsection{Monolithic approach and the FS method}
	Let us first consider a single step of the monolithic approach for the semidiscrete problem \cref{eq:biot_system}, i.e.,
	\begin{equation} \label{eq:monolithic}
		\begin{pmatrix}
			A_{uu} & -A_{up} \\ A_{up}^\top & c_fI + \Delta T A_{pp}
		\end{pmatrix} \begin{pmatrix}
		u_{n+1} \\ p_{n+1}
		\end{pmatrix} = \begin{pmatrix}
		0 & 0 \\ A_{up}^\top & c_fI
		\end{pmatrix} \begin{pmatrix}
		u_n \\ p_n
		\end{pmatrix}.
	\end{equation}

    In the following, we prove contractivity for the solution of the monolithic approach for two different norms.
	
	\begin{theorem} \label{thm:monolithic_1}
		Let us consider the monolithic approach \cref{eq:monolithic} as the coarse propagator of the Parareal algorithm. Let us further assume that matrices $A_{uu}$ and $A_{pp}$ are coercive, with coercivity constants $\alpha_{uu}$ and $\alpha_{pp}$, respectively, and that matrix $A_{up}$ satisfies the continuity condition $\|A_{up}\|_2 \leq M_{up}$. Then, the hypothesis stated in \Cref{prop:contractionv2} is satisfied for the weighted norm $\|\cdot \|_{a,b}$ for $a = c_f/2$ and $b = \alpha_{uu}$, if the following constrain is fulfilled:
		\begin{equation} \label{eq:thm_mon_0}
			\Delta T > \frac{M_{up}^2}{4 \alpha_{uu}\alpha_{pp}}.
		\end{equation}
	\end{theorem}
	\begin{proof}
        Let us consider the product between the expression \cref{eq:monolithic} and the vector $\begin{pmatrix}
			u_{n+1} \\ p_{n+1}
		\end{pmatrix}$, which leads to the equality
		\begin{equation} \label{eq:thm_mon_3}
		\begin{split}
			c_f\|p_{n+1}\|^2~ + &~ \Delta T (A_{pp}p_{n+1},p_{n+1}) + (A_{up}^\top u_{n+1},p_{n+1}) - (A_{up}p_{n+1},u_{n+1}) \\
			 + &~ (A_{uu}u_{n+1},u_{n+1}) = c_f (p_n, p_{n+1}) + (A_{up}^\top u_n, p_{n+1}).
		\end{split}
		\end{equation}

        Using the Cauchy-Schwarz and Young's inequalities, it follows
        \begin{equation} \label{eq:thm_mon_1}
            c_f(p_n, p_{n+1}) \leq \frac{c_f}{2}\|p_n\|_2^2 + \frac{c_f}{2}\|p_{n+1}\|_2^2.
        \end{equation}
        In addition, applying Young's inequality and the continuity condition for $A_{up}$, it yields
        \begin{equation} \label{eq:thm_mon_2}
            (A_{up}^\top u_n, p_{n+1}) = (u_n, A_{up}p_{n+1} ) \leq \frac{\delta}{2}\|u_n\|_2^2 + \frac{M_{up}^2}{2\delta}\|p_{n+1}\|_2^2,
        \end{equation}
        for $\delta > 0$. Then, recalling that $(A_{up}^\top u_{n+1},p_{n+1}) = (A_{up}p_{n+1},u_{n+1})$, and applying the coercivity properties of $A_{uu}$ and $A_{up}$, together with substituting \cref{eq:thm_mon_1} and \cref{eq:thm_mon_2} in equation \cref{eq:thm_mon_3}, it follows the expression
        \begin{equation*} 
			\frac{c_f}{2}\|p_{n+1}\|^2 + \Delta T \alpha_{pp}\|p_{n+1}\|^2 + \alpha_{uu} \|u_{n+1}\|^2 \leq \frac{c_f}{2}\|p_n\|^2 + \frac{\delta}{2}\|u_n\|^2 + \frac{M_{up}^2}{2\delta}\|p_{n+1}\|^2.
		\end{equation*}
        Rearranging the inequality, it yields
        \begin{equation*}
            \left( \frac{c_f}{2} + \Delta T \alpha_{pp} - \frac{M_{up}^2}{2\delta}\right)\|p_{n+1}\|^2 + \alpha_{uu}\|u_{n+1}\|^2 \leq \frac{c_f}{2}\|p_n\|^2 + \frac{\delta}{2}\|u_n\|^2.
        \end{equation*}

        Then, the solution is contractive if the coefficients from the left-hand side are greater than their counterparts from the right-hand side, that is,
        \begin{equation*}
            \Delta T \alpha_{pp} - \frac{M_{up}^2}{2\delta} > 0 \quad \text{and} \quad \alpha_{uu} \geq \frac{\delta}{2}.
        \end{equation*}
		Finally, by choosing $\delta = 2\alpha_{uu}$, from the given hypothesis \cref{eq:thm_mon_0}, it follows that the previous conditions are fulfilled. Therefore, the following inequality holds, proving the contraction:
        \begin{equation*}
		\begin{split}
			\|(u_{n+1},p_{n+1})\|_{a,b} = & ~\frac{c_f}{2} \|p_{n+1}\|_2^2 + \alpha_{uu}\|u_{n+1}\|_2^2 \\ < & ~\left( \frac{c_f}{2} + \Delta T \alpha_{pp} - \frac{M_{up}^2}{4\alpha_{uu}}\right)\|p_{n+1}\|_2^2 + \alpha_{uu}\|u_{n+1}\|_2^2 
            \\ \leq &~ \frac{c_f}{2}\|p_n\|_2^2 + \alpha_{uu}\|u_n\|_2^2 = \|(u_n,p_n)\|_{a,b}.
		\end{split}
		\end{equation*}
	\end{proof}

Observe that the contraction is guaranteed for the weighted norm only under certain constrain on the coarse time step $\Delta T$. Depending on the specific coupled problem, such a condition may not be demanding. However, we would prefer to have a contraction with respect to a norm that guarantees contraction without any extra requirements for the time step. That is the case for the energy norm $\|\cdot\|_e$.

\begin{theorem} \label{thm:monolithic_2}
    Let us consider the monolithic approach \cref{eq:monolithic} as the coarse propagator of the Parareal algorithm. Then, the hypothesis stated in \Cref{prop:contractionv2} is satisfied for the energy norm $\|\cdot\|_e$.
\end{theorem}
\begin{proof}
    Let us consider the product between the two sides of equality \cref{eq:monolithic} and the vector $\begin{pmatrix}
			u_{n+1} - u_n \\ p_{n+1}
		\end{pmatrix}$, which results in the expression
        \begin{equation} \label{eq:thm_mon_4}
            \begin{split}
                c_f\|p_{n+1}\|_2^2 & ~ + \Delta T(A_{pp}p_{n+1},p_{n+1}) + (A_{up}^\top u_{n+1},p_{n+1}) - (A_{up}p_{n+1},u_{n+1}-u_n) \\ &~+ (A_{uu}u_{n+1},u_{n+1}-u_n) = c_f(p_n,p_{n+1}) + (A_{up}^\top u_n,p_{n+1}).
            \end{split}
        \end{equation}

    First, recall that $(A_{up}^\top u_{n+1},p_{n+1})-(A_{up}^\top u_{n}, p_{n+1}) = (A_{up}p_{n+1}, u_{n+1}-u_n)$. Moreover, let us consider the identity $(x-y,x) = \frac{1}{2}\|x\|^2 + \frac{1}{2}\|x-y\|^2 - \frac{1}{2}\|y\|^2$, which holds for any induced norm. Then, it follows
    \begin{equation} \label{eq:thm_mon_5}
        (A_{uu}u_{n+1},u_{n+1}-u_n) = \frac{1}{2}\|u_{n+1}\|_{A_{uu}}^2 + \frac{1}{2}\|u_{n+1}-u_n\|_{A_{uu}}^2 - \frac{1}{2}\|u_n\|_{A_{uu}}^2.
    \end{equation}
    Substituting \cref{eq:thm_mon_1} and \cref{eq:thm_mon_5} in equation \cref{eq:thm_mon_4}, it yields the expression
    \begin{equation*}
        \begin{split}
            \frac{c_f}{2}\|p_{n+1}\|_2^ 2&~+ \Delta T\|p_{n+1}\|_{A_{pp}}^2 + \frac{1}{2}\|u_{n+1}\|_{A_{uu}}^2 \\ &~+ \frac{1}{2}\|u_{n+1}-u_n\|_{A_{uu}}^2 \leq \frac{c_f}{2}\|p_n\|_2^2 + \frac{1}{2}\|u_n\|_{A_{uu}}^2,
        \end{split}
    \end{equation*}
    where $\|\cdot\|_{A_{pp}}$ denotes the norm induced by the scalar product $(A_{pp}\cdot,\cdot)$. By assumption, $p_{n+1} \neq 0$, and the following inequality is fulfilled:
    \begin{equation*}
        \begin{split}
            \frac{1}{2}\|(u_{n+1},p_{n+1})\|_{e} = &~ \frac{c_f}{2}\|p_{n+1}\|^2_2 + \frac{1}{2} \|u_{n+1}\|_{A_{uu}}^2 \\ < &~ \frac{c_f}{2}\|p_n\|_2^2 + \frac{1}{2}\|u_n\|_{A_{uu}}^2 = \frac{1}{2}\|(u_n,p_n)\|_e,
        \end{split}
    \end{equation*}
    which proves the claim.
\end{proof}

A straightforward extension of the previous results shows that, for a sufficiently large number of iterations, the Parareal algorithm with the FS method as the coarse propagator satisfies the condition given by \Cref{prop:contractionv2} under the same hypotheses of \Cref{thm:monolithic_1} and \Cref{thm:monolithic_2}.

Let us first rewrite the FS method, described in \Cref{alg1_new}, in its matrix formulation at the $k$-th iteration and time point $t = T_ {n+1}$, that is,
\begin{equation*}
    \begin{split}
    \begin{pmatrix}
        A_{uu} & -A_{up} \\[0.25em] 0 & (c_f + L)I + \Delta T A_{pp}
    \end{pmatrix} &\begin{pmatrix}
        u_{n+1}^{k+1} \\[0.25em] p_{n+1}^{k+1}
    \end{pmatrix} \\ = \begin{pmatrix}
        0 & 0 \\[0.25em] -A_{up}^\top & LI
    \end{pmatrix} \begin{pmatrix}
        u_{n+1}^k \\[0.25em] p_{n+1}^k
    \end{pmatrix} & + \begin{pmatrix}
        0 & 0 \\[0.25em] A_{up}^\top & c_fI
    \end{pmatrix} \begin{pmatrix}
        u_n \\[0.25em] p_n
    \end{pmatrix},
    \end{split}
\end{equation*}
where, as described in \Cref{subsec:iterative}, $L > 0$ is a suitable stabilization parameter. In addition, the initial guess is generally chosen to be $u_{n+1}^0 = u_n$ and $p_{n+1}^0 = p_n$.

Prior to stating the analytical result, let us convey the notation employed hereafter. In the following, let $\Phi_{\text{mon}}$ denote the time step matrix of the monolithic approach, whereas $\Phi_{\text{FS}}$ stands for the iteration matrix of the FS method. Then, the following result can be proven.

\begin{corollary} \label{cor:FS}
    Let us consider the FS method (\Cref{alg1_new}) as the coarse propagator within the Parareal algorithm. Let us further assume that the monolithic method is contractive, that is, the inequality 
    \begin{equation} \label{eq:hyp_FS}
        \|\Phi_{\text{mon}}(u,p)\| <\|(u,p)\|
    \end{equation}
    holds for any nonzero vector $(u,p)$ and a suitable vector norm induced by a scalar product. Then, for a sufficiently large number of iterations $k_{\text{max}}$, there holds
    \begin{equation*}
        \|\Phi_{\text{FS}}^{k_{\text{max}}}(u,p)\| < \|(u,p)\|
    \end{equation*}
    for any nonzero vector $(u,p)$.
\end{corollary}
\begin{proof}
    By the hypothesis \cref{eq:hyp_FS}, there must exist a positive scalar $\varepsilon > 0$ such that
    \begin{equation} \label{eq:cor_FS_1}
        \|\Phi_{\text{mon}}(u,p)\| + \varepsilon = \|(u,p)\|.
    \end{equation}
    On the other hand, since the FS method converges to the solution provided by the monolithic approach, there exists $k_{\text{max}} \in \mathbb{N}$ such that
    \begin{equation} \label{eq:cor_FS_2}
        \|\Phi_{\text{FS}}^{k_{\text{max}}}(u,p) - \Phi_{\text{mon}}(u,p)\| < \varepsilon.
    \end{equation}
    Then, it is straightforward to show the following inequality:
    \begin{equation*}
        \|\Phi_{\text{FS}}^{k_{\text{max}}}(u,p)\| < \|\Phi_{\text{mon}}(u,p)\| + \varepsilon =\|(u,p)\|,
    \end{equation*}
    which completes the proof.
\end{proof}

\begin{remark}
    Note that, from \cref{eq:cor_FS_1}, it yields that the more the approximation for the monolithic approach contracts, the less restrictive the condition \cref{eq:cor_FS_2} will be, enabling the value of $k_{\text{max}}$ to be smaller.
\end{remark}

\subsection{Multirate FS method}
As explained in \Cref{subsec:iterative}, by employing multirate, each time step integration of the elliptic equation involves $q$ consecutive numerical integrations of the parabolic equation. Therefore, both matrix formulation and the definition of the weighted norm \cref{eq:weighted_norm} vary slightly with respect to the single rate setting.

Let us first formulate the direct method given by \cref{eq:fully_discrete_problem_multirare} and the multirate FS method (cf. \Cref{alg_2_new}) in their matrix formulation. For that purpose, we introduce new notation for the sake of compactness. In particular, let us consider the notation
\begin{equation*}
    \boldsymbol{P}_{n,q} = \begin{pmatrix}
        p_{n+1} \\ p_{n+2} \\ \vdots \\ p_{n+q}
    \end{pmatrix}, \quad A_{11} = I_q \otimes (c_f q I + \Delta T A_{pp})-c_f q D_q\otimes I,
\end{equation*}
$A_{01} = -e_q^\top \otimes A_{up}$, $A_{10} = \boldsymbol{1}_q \otimes A_{up}^\top$, and $B_{11} = c_f q e_1$, where $I_q$ is the $q\times q$ identity matrix, $D_q$ represents the matrix of size $q\times q$ with entries at the diagonal below the main diagonal equal to 1, and zero otherwise, $\boldsymbol{1}_q$ denotes the vector of size $q$ with all the entries equal to one, and $\{e_j\}_{j\in\{1,\dots,q\}}$ is the canonical basis of the space $\mathbb{R}^q$. Then, the matrix formulation of the method \cref{eq:fully_discrete_problem_multirare} at $t = T_{n+q}$ reads:
\begin{equation} \label{eq:matrix_multirate_direct}
    \begin{pmatrix}
        A_{uu} & A_{01} \\ A_{10} & A_{11}
    \end{pmatrix} \begin{pmatrix}
         u_{n+q} \\ \boldsymbol{P}_{n,q}
    \end{pmatrix} = \begin{pmatrix}
        0 & 0 \\ A_{10} & B_{11}
    \end{pmatrix} \begin{pmatrix}
        u_n \\ p_n
    \end{pmatrix}.
\end{equation}

In turn, for the matrix formulation of the multirate FS iterative scheme, we further introduce the following matrices:
\begin{equation*}
    A_{11}^L = I_q \otimes (q(c_f + L)I + \Delta T A_{pp}) - q(c_f+L)D_q \otimes I
\end{equation*}
and $B_{11}^L = qL(I_q - D_q)\otimes I$, which lead us to the following formulation for the $i$-th iteration of the multirate FS method at $T = T_{n+q}$:
\begin{equation} \label{eq:mat_multirate}
    \begin{pmatrix}
        A_{uu} & A_{10} \\ 0 & A_{11}^L
    \end{pmatrix} \begin{pmatrix}
         u^i_{n+q} \\ \boldsymbol{P}^i_{n,q}
    \end{pmatrix} = \begin{pmatrix}
        0 & 0 \\ -A_{10} & B_{11}^L
    \end{pmatrix} \begin{pmatrix}
         u^{i-1}_{n+q} \\ \boldsymbol{P}^{i-1}_{n,q}
    \end{pmatrix} + \begin{pmatrix}
        0 & 0 \\ A_{10} & B_{11}
    \end{pmatrix} \begin{pmatrix}
        u_n \\ p_n
    \end{pmatrix}.
\end{equation}

Likewise, let us define a new weighted norm $\|\cdot\|_{a,\textbf{b}}$, where $a \in \mathbb{R}_+$ and $\textbf{b} = (b_1,\dots,b_q) \in \mathbb{R}^q_+$. In particular, the norm $\|\cdot \|_{a,\textbf{b}}$ is defined by
\begin{equation*}
    \|(u_{n+q},\boldsymbol{P}_{n,q})\|_{a,\textbf{b}} =  a\|u_{n+q}\|_2^2 + \sum_{j=1}^q b_j\|p_{n+j}\|_2^2.
\end{equation*}
Observe that the norm comprises now the values of the unknowns $p_{n+j}$ at the intermediate points where only the parabolic equation is evaluated.

In this section, we proceed as in the previous one, that is, we first show the constrains that enable the Parareal algorithm with the method \cref{eq:matrix_multirate_direct} as coarse propagator to achieve convergence. Then, the result is extended to the iterative counterpart \cref{eq:mat_multirate}.

\begin{theorem} \label{thm:multirate}
    Let us consider the multirate method \cref{eq:matrix_multirate_direct} as the coarse propagator within the Parareal algorithm. Let us further assume that $A_{uu}$ and $A_{pp}$ are coercive, with constants $\alpha_{uu}$ and $\alpha_{pp}$, respectively, and that $A_{up}$ satisfies the continuity condition $\|A_{up}\|_2 \leq M_{up}$. Then, the hypothesis stated in \Cref{prop:contractionv2} is fulfilled for the weighted norm $\|\cdot\|_{\textbf{a},b}$ for $a = \frac{2q-1}{4q-3}\alpha_{uu}$, $b_j = \varepsilon$, with $j \in \{1,\dots,q-1\}$, and $b_q = \frac{q c_f}{2}$, if the following constrain holds:
    \begin{equation} \label{eq:thm_multirate_0}
        \Delta T > \frac{(4q-3)M_{up}^2}{4\alpha_{uu}\alpha_{pp}},
    \end{equation}
    where $\displaystyle \varepsilon = \Delta T \alpha_{pp} - \frac{4q-3}{4\alpha_{uu}}M_{up}^2$.
\end{theorem}
\begin{proof}
    Let us take the product between the expression \cref{eq:matrix_multirate_direct} and the vector $\displaystyle \begin{pmatrix}
        u_{n+q} \\ \boldsymbol{P}_{n,q}
    \end{pmatrix}$, yielding the equality
    \begin{equation} \label{eq:thm_multirate_1}
        \begin{split}
            (A_{uu}u_{n+q},u_{n+q}) & + \sum_{j=1}^q (A_{up} p_{n+j},u_{n+q}-u_n) - (A_{up}p_{n+q},u_n) \\ & + c_f q \sum_{j=1}^q (p_{n+j}-p_{n+j-1},p_{n+j}) + \Delta T \sum_{j=1}^q (A_{pp}p_{n+j},p_{n+j}) = 0.
        \end{split}
    \end{equation}

    By the property $(x-y,x)=\frac{1}{2}\|x\|_2^2 + \frac{1}{2}\|x-y\|_2^2 - \frac{1}{2}\|y\|_2^2$, it follows
    \begin{equation} \label{eq:thm_multirate_2}
        \begin{split}
            \sum_{j=1}^q (p_{n+j}-p_{n+j-1},p_{n+j}) & = \frac{1}{2}\sum_{j=1}^q(\|p_{n+j}\|_2^2 + \|p_{n+j} - p_{n+j-1}\|_2^2 - \|p_{n+j-1}\|_2^2) \\ & \geq \frac{1}{2}\sum_{j=1}^q (\|p_{n+j}\|_2^2 - \|p_{n+j-1}\|_2^2) = \frac{1}{2}(\|p_{n+q}\|_2^2 - \|p_n\|_2^2),
        \end{split}
    \end{equation}
    noting that the defined sum is in fact a telescopic sum.

    On the other hand, applying Young's inequality and continuity properties, it yields
    \begin{equation} \label{eq:thm_multirate_3}
        \sum_{j=1}^{q-1} (A_{up} p_{n+j}, u_n - u_{n+q}) \leq \sum_{j=1}^{q-1} \left( \frac{M_{up}^2}{2\delta} \|p_{n+j}\|_2^2 + \delta \left(\|u_n\|_2^2 + \|u_{n+q}\|_2^2\right)\right),
    \end{equation}
    where $\delta > 0$ is a tuning parameter. Observe that \Cref{lemma} has been applied to the term $\|u_n - u_{n+q}\|_2^2$. Similarly, the following inequality can be concluded:
    \begin{equation} \label{eq:thm_multirate_4}
        (A_{up} p_{n+q},u_n) \leq \frac{M_{up}^2}{2\delta} \|p_{n+q}\|_2^2 + \frac{\delta}{2}\|u_n\|_2^2.
    \end{equation}

    Then, applying the coercivity properties of $A_{uu}$ and $A_{pp}$, together with replacing expressions \cref{eq:thm_multirate_2}, \eqref{eq:thm_multirate_3}, and \cref{eq:thm_multirate_4} in equation \cref{eq:thm_multirate_1}, it follows
    \begin{equation*}
        \begin{split}
            (\alpha_{uu} - \delta(q &-1))\|u_{n+q}\|_2^2 + \left( \Delta T \alpha_{pp} - \frac{M_{up}^2}{2\delta} \right) \sum_{j=0}^{q-1} \|p_{n+j}\|_2^2 \\ &+ \left( \frac{c_fq}{2} + \Delta T \alpha_{pp} - \frac{M_{up}^2}{2\delta} \right) \|p_{n+q}\|_2^2 \leq \frac{c_fq}{2}\|p_n\|_2^2 + \frac{2q-1}{2}\delta \|u_n\|_2^2.
        \end{split}
    \end{equation*}

    The solution is contractive if the coefficients from the left-hand side are greater than their counterparts from the right-hand side, i.e., 
    \begin{equation*}
        \Delta T \alpha_{pp} - \frac{M_{up}^2}{2\delta} > 0 \quad \text{and} \quad \alpha_{uu} - \delta (q-1) \geq \frac{2q-1}{2}\delta.
    \end{equation*}
    In order to hold the second inequality, let us set $\displaystyle \delta = \frac{2\alpha_{uu}}{4q-3}$, and by the hypothesis \cref{eq:thm_multirate_0}, it is concluded that the first inequality also holds. Thus, the following inequality yields
    \begin{equation*}
        \begin{split}
            \|(u_{n+q}, \boldsymbol{P}_{n,q})\|_{a,\textbf{b}} & = \frac{2q-1}{4q-3} \alpha_{uu} \|u_{n+q}\|_2^2 + \sum_{j=1}^{q-1}\varepsilon \|p_{n+j}\|_2^2 + \frac{q c_f}{2}\|p_{n+q}\|_2^2 \\
            & < \frac{2q-1}{4q-3} \alpha_{uu} \|u_n\|_2^2 + \frac{q c_f}{2}\|p_{n}\|_2^2 \leq \|(u_{n}, \boldsymbol{P}_{n-1,q})\|_{a,\textbf{b}},
        \end{split}
    \end{equation*}
    which proves the claim.
\end{proof}

\begin{remark}
    Observe that, if we consider $q = 1$ (that is, the single rate monolithic method), the inequality \cref{eq:thm_multirate_0} is equivalent to the one stated in \cref{eq:thm_mon_0} for the monolithic approach.
\end{remark}

\begin{remark}
    The larger we set $q$, the more restrictive the condition \cref{eq:thm_multirate_0} becomes. However, we must highlight that $\Delta T$ denotes the time step assigned to the elliptic problem, whereas the time step employed for the integration of the parabolic equation is given by $\Delta T/q$. This means that, even though $\Delta T$ should be increased for larger $q$ in order to guarantee convergence, the finer time steps employed by the parabolic equation do not change as abruptly as $\Delta T$.
\end{remark}

Finally, it can be proven that, for a sufficiently large number of iterations, the multirate FS method chosen as the coarse propagator within the Parareal algorithm guarantees convergence for the same norm as the one defined in \Cref{thm:multirate}. This result is based on the idea that the multirate FS method converges to the method \cref{eq:matrix_multirate_direct}. Before stating the result, let us denote by $\Phi_{\text{MD}}$ the time iteration matrix obtained for the method \cref{eq:matrix_multirate_direct}, whereas $\Phi_{\text{MFS}}$ denotes the iteration matrix of the multirate FS method \cref{eq:mat_multirate}.

\begin{corollary}
    Let us consider the multirate FS method (\Cref{alg_2_new}) as the coarse propagator within the Parareal algorithm. Let us further assume that the multirate method \cref{eq:fully_discrete_problem_multirare} is contractive, that is, the inequality
    \begin{equation*}
        \|\Phi_{\text{MD}}(u,\boldsymbol{P})\| < \|(u,\boldsymbol{P})\|
    \end{equation*}
    holds for any nonzero vector $(u,\boldsymbol{P})$ and a suitable vector norm induced by a scalar product. Then, for a sufficiently large number of iterations $k_{\text{max}}$, there holds
    \begin{equation*}
        \|\Phi_{\text{MFS}}^{k_{\text{max}}} (u,\boldsymbol{P})\| < \|(u,\boldsymbol{P})\|,
    \end{equation*}
    for any nonzero vector $(u,\boldsymbol{P})$.
\end{corollary}
\begin{proof}
    The proof follows exactly the same procedure as the one for \Cref{cor:FS}.
\end{proof}

Therefore, with the restriction \cref{eq:thm_multirate_0} and a sufficiently large number of iterations for the FS iterative scheme, the Parareal algorithm with th multirate FS method converges for the weighted norm defined in \Cref{thm:multirate}.
    
\section{Numerical experiments} \label{sec:experiments}

The main goal of this section is twofold. On the one hand, we validate the qualitative results obtained in the previous section. In addition, we provide further convergence outcomes when the methods are implemented for the Biot's consolidation model \cref{eq:nlBiot} in more complex settings. Two benchmark problems are considered for that purpose, namely, a simple manufactured problem and the well-known Mandel's problem (see \cite{Mandel1953}). The same spatial discretization is employed for both test problems, that is, the stabilized P1-P1 scheme (see \cite{Aguilar_et_al_2008,Rodrigo_et_al_2016}) over a regular triangular mesh.

In the following, we introduce the three proper solvers under consideration in this section. In particular, three combinations are explored, differing in the propagators considered by the Parareal algorithm. Let us denote by Parareal/FS-FS the solver that considers the FS method as both propagators. Besides, Parareal/FS-MFS represents the solver constructed by the FS and multirate FS methods as coarse and fine propagators, respectively. Finally, the solver that considers the multirate FS method as both propagators is denoted by Parareal/MFS-MFS. In terms of notation, $q_c$ and $q_f$ denote the number of local flow time steps within one mechanics time step for the coarse and fine propagators, respectively. Remarkably, the energy norm is employed in order to compute the error of the proposed solvers, although similar convergence results are achieved for the weighted norms defined in the previous section. 

\subsection{Manufactured problem}

Let us first consider a rather simple problem, whose analytic solution is given by
\begin{subequations}\label{eq:exact_sol_manufact}
\begin{align}
u(x,y,t) &= \begin{pmatrix} \sin(\pi x)\cos(\pi y) e^{-t} \\ \cos(\pi x)\sin(\pi y) e^{-t}\end{pmatrix}, \\
p(x,y,t) &= \frac{4\mu + 2\lambda}{\alpha}\pi \cos(\pi x)\cos(\pi y)e^{-t},
\end{align}
\end{subequations}
in the domain $\Omega \times [0,T] = (0,1)^2 \times [0,5]$ and with $g \equiv 0$ and suitable $f$ and boundary and initial conditions in order to ensure that \cref{eq:exact_sol_manufact} is the solution of the problem \cref{eq:nlBiot}. We further set $\mu = \lambda = 10^4$, $\alpha = 1$, $S = 10^{-3}$, and $\Delta x = \Delta y = 0.25$. 

Then, \Cref{fig:error_Exp1} reports the convergence rate of the three proposed solvers for different values of permeability $K$. We consider $\delta t = 0.01$, $\Delta T = 0.1$, $q_f = 2$, and $q_c = 50$. The rate is computed dividing the error at every iteration by the error at the first iteration. It is worth noting that the three solvers converge rapidly to their fine solutions. In particular, it is remarkable that it only takes four iterations for the Parareal/MFS-MFS method to converge, which makes the solver especially useful. At first, the problems with lower permeabilities converge faster in all cases except for the Parareal/MFS-MFS, but after certain number of iterations, the tendency appears to change, achieving better convergence results for the cases with larger permeability. These results go hand in hand with the bounds provided in the previous section, which suggest that larger permeabilities are more likely to converge.

\begin{figure}[tbhp]
	\centering
	\subfloat[Parareal/FS-FS]{\label{fig:FS_FS_Exp1}\includegraphics[width=0.465\textwidth]{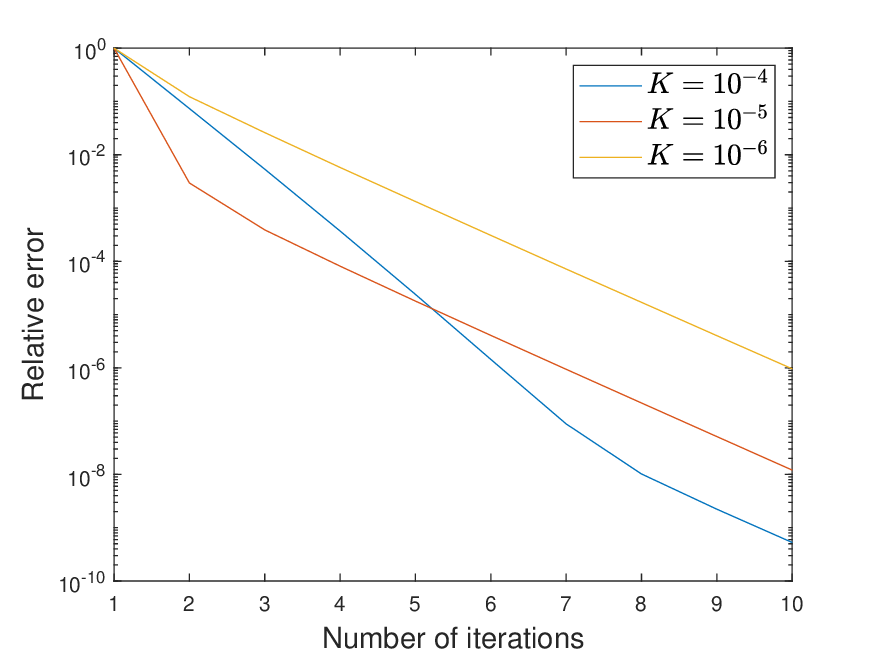}}
	\subfloat[Parareal/FS-MFS]{\label{fig:FS_FS_Exp1}\includegraphics[width=0.465\textwidth]{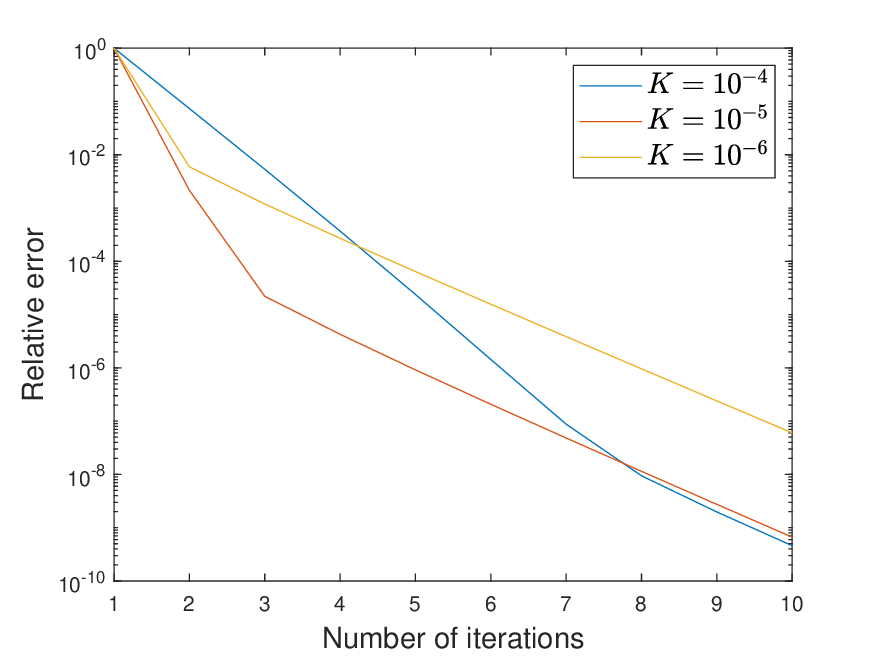}}
	
	\subfloat[Parareal/MFS-MFS]{\label{fig:MFS_MFS_Exp1}\includegraphics[width=0.465\textwidth]{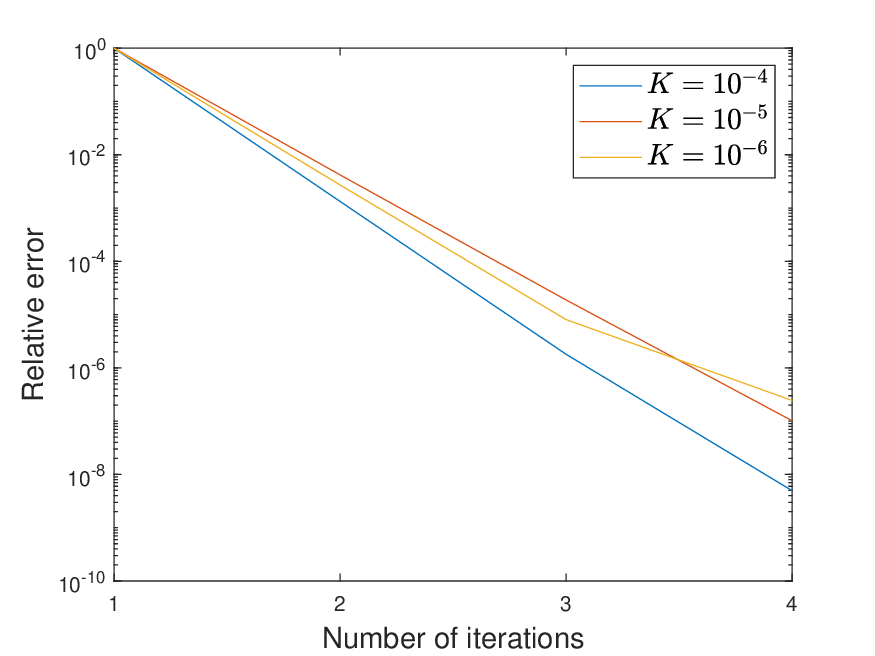}}
	\caption{Relative error of the Parareal/FS-FS, the Parareal/FS-MFS, and the Parareal/MFS-MFS methods applied to the manufactured problem, for decreasing value of $K$ and an increasing number of iterations.}
	\label{fig:error_Exp1}
\end{figure}

\subsection{Mandel's problem}

First formulated in \cite{Mandel1953}, the Mandel's problem is described as a poroelastic slab of size $2a$ and $2b$ in the $x$ and $y$ directions, respectively, and infinitely long in the $z$ direction. The slab is enclosed between two rigid plates that, at $t = 0$, compress the slab with an equal force of magnitude $2F$. Furthermore, the slab is considered to remain drained at $x = \pm a$. The physical representation of the problem is sketched in \Cref{fig:mandel}.

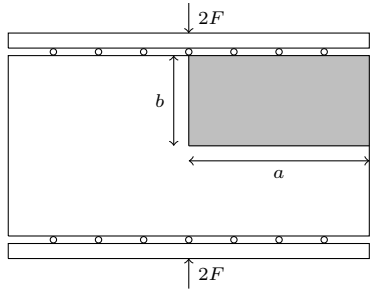
\begin{figure}[tbhp]
    \centering
    	\begin{tikzpicture}[scale = 0.4]
			\filldraw[lightgray,draw opacity=0.25] (6,3.75) -- (12,3.75) -- (12,6.75) -- (6,6.75) -- (6,3.75);
			
			\draw (0,0) -- (12,0) -- (12,0.5) -- (0,0.5) -- (0,0);
			\draw (0,0.75) -- (12,0.75) -- (12,6.75) -- (0,6.75) -- (0,0.75);
			\draw (0,7) -- (12,7) -- (12,7.5) -- (0,7.5) -- (0,7);
			\draw (6,3.75) -- (12,3.75);
			\draw (6,3.75) -- (6,6.75);
			
			\draw (1.5,0.625) circle (3pt);
			\draw (3,0.625) circle (3pt);
			\draw (4.5,0.625) circle (3pt);
			\draw (6,0.625) circle (3pt);
			\draw (7.5,0.625) circle (3pt);
			\draw (9,0.625) circle (3pt);
			\draw (10.5,0.625) circle (3pt);
			\draw (1.5,6.875) circle (3pt);
			\draw (3,6.875) circle (3pt);
			\draw (4.5,6.875) circle (3pt);
			\draw (6,6.875) circle (3pt);
			\draw (7.5,6.875) circle (3pt);
			\draw (9,6.875) circle (3pt);
			\draw (10.5,6.875) circle (3pt);
			
			\draw[->] (6,-1) -- (6,0);
			\draw[->] (6,8.5) -- (6,7.5);
			
			\draw[<->] (5.5,3.75) -- (5.5,6.75);
			\draw[<->] (6,3.25) -- (12,3.25);
			
			\path [anchor=west] (6,-0.5) node {\scriptsize $2F$};
			\path [anchor=west] (6,8) node {\scriptsize $2F$};
			\path [anchor=east] (5.5,5.25) node {\scriptsize $b$};
			\path [anchor=north] (9,3.25) node {\scriptsize $a$};
			
		\end{tikzpicture}
    \caption{Mandel's problem.}
    \label{fig:mandel}
\end{figure}

Since the problem is symmetric in the $x$ and $y$ directions, and independent of the $z$ direction, we restrict the problem to a quarter of its original domain, i.e., $\Omega = (0,a) \times (0,b)$. The boundary conditions of the problem are thus displayed in \Cref{tab:bc_mandel}, while the initial conditions are imposed to be the exact solution in \cite{MikelicWangWheeler2014} at $t = 0$. Observe that $u_y(b,t)$ represents the exact solution of the displacement in the $y$-direction at $y = b$. Finally, the input parameters for the problem are provided by \Cref{tab:input_mandel}.

\begin{table}[tbhp]
\footnotesize
\caption{Boundary conditions for the Mandel's problem.}\label{tab:bc_mandel}
\begin{center}
\begin{tabular}{|c|c|c|} \hline
Boundary & Flow & Mechanics \\ \hline
$x = 0$ & $\nabla p\cdot \boldsymbol{n} = 0$ & $u \cdot \boldsymbol{n} = 0$ \\
$y = 0$ & $\nabla p\cdot \boldsymbol{n} = 0$ & $u \cdot \boldsymbol{n} = 0$ \\
$x = a$ & $ p = 0$ & $\sigma \cdot \boldsymbol{n} = 0$ \\
$y = b$ & $\nabla p\cdot \boldsymbol{n} = 0$ & $\sigma_{12} = 0, ~~ u\cdot\boldsymbol{n} = u_y(b,t)$ \\
\hline
\end{tabular}
\end{center}
\end{table}

\begin{table}[tbhp]
\footnotesize
\caption{Input parameters for the Mandel's problem.}\label{tab:input_mandel}
\begin{center}
\begin{tabular}{|ll|} \hline
Dimension in $x$ ($a$): & 100 m \\
Dimension in $y$ ($b$): & 10 m \\
Young's modulus ($E$): & $5.94\cdot 10^{9}$ Pa \\
Poisson's ratio ($\nu$): & 0.2 \\
Fluid compressibility ($c_f$): & $3.03\cdot 10^{-10}$ /Pa \\
Biot's constant ($\alpha$): & 1.0 \\
Initial porosity ($\phi$): & 0.2 \\
Fluid viscosity ($\eta$): & 1.0 cP \\
Grid spacing in $x$ ($\Delta x$) & 2.5 m \\
Grid spacing in $y$ ($\Delta y$): & 0.25 m \\
Total simulation time ($T$): & 50,000 s \\
Skempton coefficient ($B$): & 0.83333 \\
Undrained Poisson's ratio ($\nu_u$): & 0.44 \\
Biot's modulus ($M$): & $1.65\cdot 10^{10}$ Pa \\
\hline
\end{tabular}
\end{center}
\end{table}

Then, \Cref{fig:error_Exp2} shows the relative error of the three solvers when they are implemented for solving the Mandel's problem with several values of $K$, and numerical parameters $\delta t = 10$ s, $\Delta T = 200$ s, $q_f = 2$, and $q_c = 20$. As for the manufactured problem, all the methods converge for every iteration, and the convergence speed is remarkably high. In this case, both Parareal/FS-FS and Parareal/FS-MFS methods have an increasing convergence rate when $K$ is set to be larger, except for $K = $ 1000 mD for the former solver, where the convergence is slightly slower than for the other values of the permeability. However, that behaviour no longer applies to the Parareal/MFS-MFS method, for which just the opposite can be observed, that is, lower values of $K$ give a better convergence performance. Such a phenomenon may be caused by the choice of the multirate iterative scheme as the coarse propagator of the Parareal algorithm, although for the manufactured problem that behaviour was not observed. The Parareal/FS-FS method provides the slowest convergence among the three proposed solvers, which highlights the potential of the multirate techniques in combination with the Parareal algorithm.

\begin{figure}[tbhp]
	\centering
	\subfloat[Parareal/FS-FS]{\label{fig:FS_FS_Exp2}\includegraphics[width=0.465\textwidth]{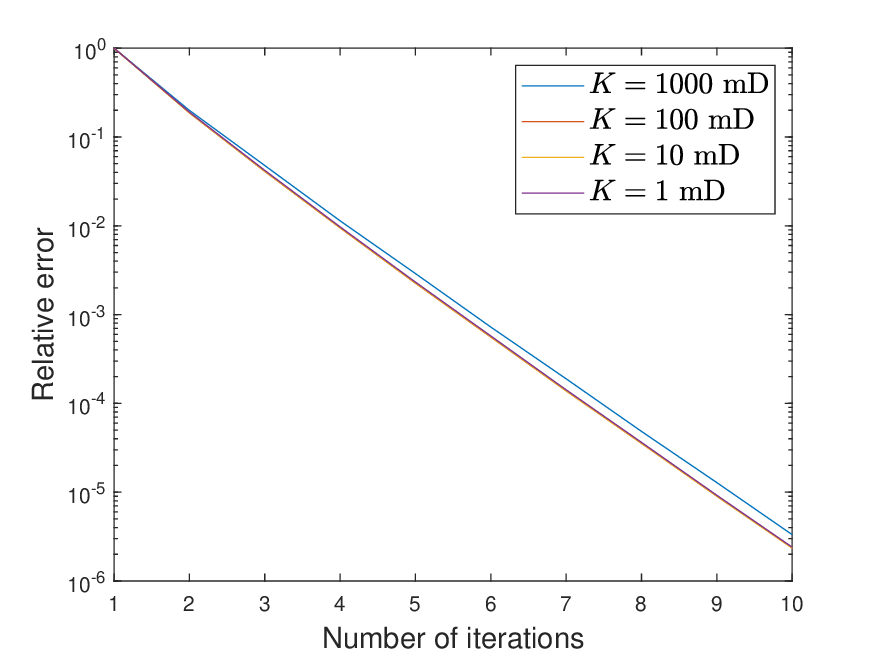}}
	\subfloat[Parareal/FS-MFS]{\label{fig:FS_FS_Exp2}\includegraphics[width=0.465\textwidth]{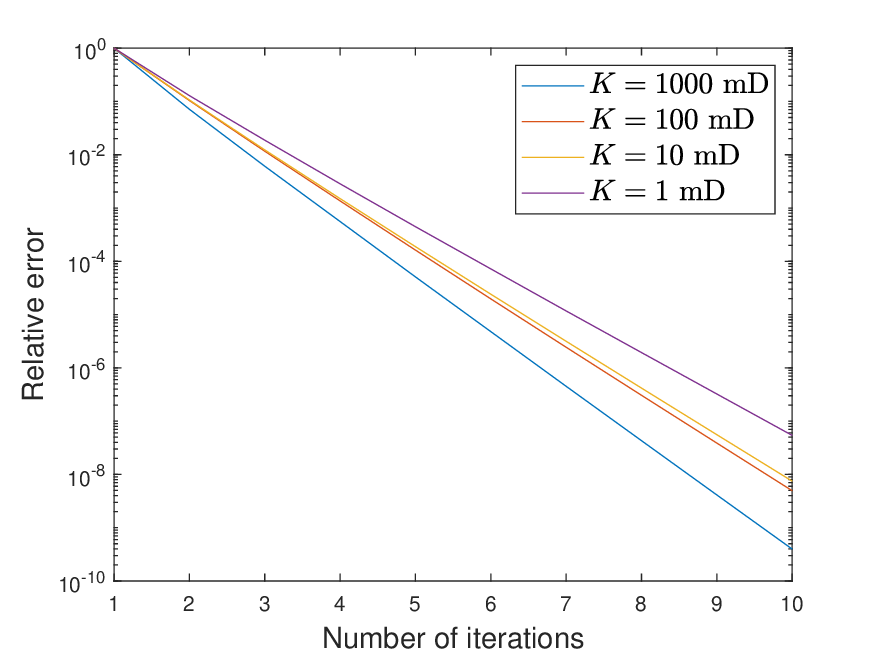}}
	
	\subfloat[Parareal/MFS-MFS]{\label{fig:MFS_MFS_Exp2}\includegraphics[width=0.465\textwidth]{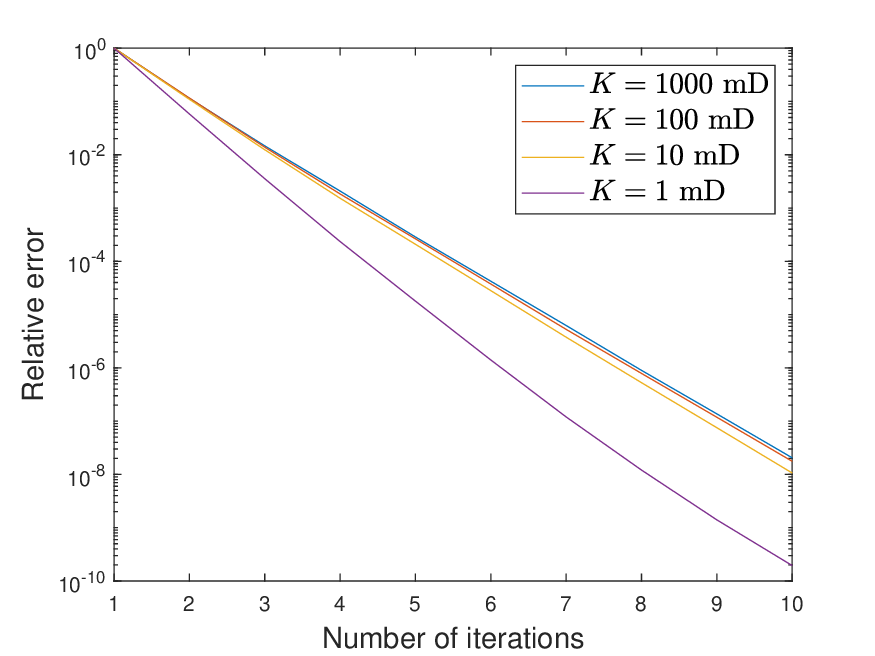}}
	\caption{Relative error of the Parareal/FS-FS, the Parareal/FS-MFS, and the Parareal/MFS-MFS methods applied to the Mandel's problem, for decreasing value of $K$ and an increasing number of iterations.}
	\label{fig:error_Exp2}
\end{figure}

\section{Conclusions} \label{sec:conclusions}

We have extended the Parareal algorithm to the case of degenerate differential-algebraic equations as given by the Biot equations. We establish sufficient conditions for convergence of the Parareal algorithm under each coarse propagator, including explicit time step restrictions derived from coercivity and continuity properties of the system operators. The multirate FS scheme exploits the separation of time scales between flow and mechanics, enabling asynchronous time stepping and improved computational efficiency. The matrix-based formulation  for this scheme provides a compact and generalizable framework for future extensions.

Numerical experiments on both a manufactured solution and the Mandel benchmark problem validated the theoretical predictions and demonstrated the robustness and convergence of the proposed methods across a range of permeability values. The results highlight the potential of multirate and time-parallel strategies for accelerating simulations in poromechanics and other multiphysics applications.


\section*{Acknowledgments}
IJC would like to acknowledge the financial support from Public University of Navarre (PhD Grant). The work of IJC and FG was supported by Grant PID2022-140108NB-I00 funded by MCIN/AEI/10.13039/501100011033 and by ``ERDF A way of making Europe''. FR and KK would like to acknowledge the financial support from the Vista Center for Modeling of Coupled Subsurface Dynamics (VISTA CSD). KK would like to acknowledge the Center for Sustainable Subsurface Resources (CSSR). FR wants to thank the support from the project MUPSI, CETP-2023-00298. 

\bibliographystyle{siamplain}
\bibliography{references}
\end{document}